\documentclass[dvipdfmx]{amsart}

\usepackage{amssymb}

\usepackage{xcolor}

\usepackage{amsfonts}
\usepackage{mathrsfs}
\usepackage{bbm}

\theoremstyle{plain}
    \newtheorem{theorem}{Theorem}[section]
    \newtheorem{lemma}[theorem]{Lemma}
    \newtheorem{proposition}[theorem]{Proposition}
    \newtheorem{corollary}[theorem]{Corollary}
\theoremstyle{definition}
    \newtheorem{definition}{Definition}[section]
    \newtheorem{remark}{Remark}[section]
    
\theoremstyle{remark}
    \newtheorem*{notation}{Notation}

\numberwithin{equation}{section}

\renewcommand{\l}{\left}
\renewcommand{\r}{\right}
\newcommand{\cleq}{\lesssim}
\newcommand{\cgeq}{\gtrsim}
\newcommand{\ceq}{\approx} 
\def\norm#1{\left\Vert #1 \right\Vert} 
\def\jbra#1{\left\langle #1 \right\rangle} 
\def\tbra#1#2{\left\langle #1 , #2 \right\rangle} 

\newcommand{\1}{\mathbbm{1}}

\newcommand{\C}{\mathbb{C}}

\newcommand{\R}{\mathbb{R}}
\renewcommand{\S}{\mathbb{S}}

\newcommand{\Z}{\mathbb{Z}}

\newcommand{\cD}{\mathcal{D}}

\newcommand{\cF}{\mathcal{F}}

\newcommand{\cS}{\mathcal{S}}

\newcommand{\cW}{\mathcal{W}}

\begin{document}

\title[Endpoint Strichartz est. for DW and its application]{Endpoint Strichartz estimate for the damped wave equation and its application}

\author[T.~Inui]{Takahisa Inui}
\address[T.~Inui]{
Department of Mathematics, 
Graduate School of Science, 
Osaka University, 
Toyonaka, Osaka 560-0043, Japan
}
\email{inui@math.sci.osaka-u.ac.jp}

\author[Y.~Wakasugi]{Yuta Wakasugi}
\address[Y.~Wakasugi]{
Department of Engineering for Production and Environment,
Graduate School of Science and Engineering,
Ehime University,
3 Bunkyo-cho, Matsuyama, Ehime, 790-8577, Japan
}
\email{wakasugi.yuta.vi@ehime-u.ac.jp}

\date{\today}

\maketitle

\begin{abstract}
Recently, the Strichartz estimates for the damped wave equation was obtained by the first author \cite{Inup} except for the wave endpoint case. In the present paper, we give the Strichartz estimate in the wave endpoint case. 
We slightly modify the argument of Keel--Tao \cite{KeTa98}. 
Moreover, we apply the endpoint Strichartz estimate to the unconditional uniqueness for the energy critical nonlinear damped wave equation. This problem seems not to be solvable as the perturbation of the wave equation. 
\end{abstract}

\tableofcontents


\section{Introduction}

\subsection{Background}

We consider the damped wave equation.
\begin{align}
\label{DW}
\left\{
\begin{array}{ll}
	\partial_t^2 \phi - \Delta \phi +\partial_t \phi = 0, & (t,x) \in (0,\infty) \times \mathbb{R}^d,
	\\
	(\phi(0),\partial_t \phi(0)) = (\phi_0,\phi_1), & x\in \mathbb{R}^d,
\end{array}
\right.
\end{align}
where $d \in \mathbb{N}$, $(\phi_0,\phi_1)$ is given, and $\phi$ is an unknown complex valued function. 
Matsumura \cite{Mat76} applied the Fourier transform to \eqref{DW} and obtained the formula
\begin{align*}
	\phi(t,x)=\cD(t) (\phi_0+\phi_1) +\partial_t \cD(t) \phi_0,
\end{align*}
where $\cD(t)$ is defined by
\begin{align*}
	\cD(t):=e^{-\frac{t}{2}}\cF^{-1} L(t,\xi) \cF
\end{align*} 
with 
\begin{align*}
	L(t,\xi):=
	\left\{
	\begin{array}{ll}
	\displaystyle
	\frac{\sinh(t \sqrt{1/4-|\xi|^2})}{\sqrt{1/4-|\xi|^2}} & \text{if } |\xi| \leq 1/2,
	\\
	\ 
	\\
	\displaystyle
	\frac{\sin(t \sqrt{|\xi|^2-1/4})}{\sqrt{|\xi|^2-1/4}} & \text{if } |\xi|>1/2.
	\end{array}
	\right.
\end{align*}
Matsumura \cite{Mat76} also obtained an $L^p$-$L^q$ type estimate. After his work, many researchers showed  $L^p$-$L^q$ type estimates  (see e.g. \cite{Nis03,HoOg04,Nar05,SaWa17,IIOW19} and references therein). 

In \cite{Wat17}, Watanabe discussed the Strichartz estimates for \eqref{DW} when $d=2,3$. 
Recently, the first author \cite{Inup} obtained the following homogeneous and inhomogeneous Strichartz estimates.

\begin{proposition}[Homogeneous Strichartz estimates]
\label{prop1.1}
Let $d \geq 2$, $2 \leq r < \infty$, and $2\leq q \leq \infty$. Set $\gamma:= \max\{ d(1/2 - 1/r)-1/q, \frac{d+1}{2}(1/2-1/r)\}$. Assume 
\begin{align*}
	\frac{d}{2}\l( \frac{1}{2} - \frac{1}{r}\r) \geq \frac{1}{q},
\end{align*}
Then, we have
\begin{align*}
	\norm{\cD(t)f}_{L_{t}^{q} (I: L_{x}^{r}(\R^d))} 
	&\cleq \norm{ \jbra{\nabla}^{\gamma-1} f}_{L^2},
	\\
	\norm{\partial_t \cD(t)f}_{L_{t}^{q} (I: L_{x}^{r}(\R^d))} 
	&\cleq \norm{ \jbra{\nabla}^{\gamma} f}_{L^2},
	\\
	\norm{\partial_t^2 \cD(t)f}_{L_{t}^{q} (I: L_{x}^{r}(\R^d))} 
	&\cleq \norm{ \jbra{\nabla}^{\gamma+1} f}_{L^2}.
\end{align*}
\end{proposition}

\begin{proposition}[Inhomogeneous Strichartz estimates]
\label{prop1.2}
Let $d\geq2$, $2\leq r,\tilde{r} < \infty$, and $2\leq q, \tilde{q} \leq \infty$. We set $\gamma:= \max \{ d(1/2-1/r)-1/q, \frac{d+1}{2}(1/2-1/r) \}$ and $\tilde{\gamma}$ in the same manner. Assume that $(q,r)$ and $(\tilde{q},\tilde{r})$ satisfies 
\begin{align*}
	\frac{d}{2}\l( \frac{1}{2} - \frac{1}{r}\r) +\frac{d}{2}\l( \frac{1}{2} - \frac{1}{\tilde{r}}\r)
	>\frac{1}{q} + \frac{1}{\tilde{q}},
\end{align*}
\begin{align*}
	\frac{d}{2} \l( \frac{1}{2} - \frac{1}{r}\r) + \frac{d}{2} \l( \frac{1}{2} - \frac{1}{\tilde{r}}\r) 
	= \frac{1}{q} + \frac{1}{\tilde{q}}
	\text{ and }
	1< \tilde{q}' < q<\infty,
\end{align*}
or 
\begin{align*}
	(q,r)=(\tilde{q},\tilde{r})=(\infty,2).
\end{align*}
Moreover, we exclude the endpoint case, that is, we assume $(q,r) \neq (2,2(d-1)/(d-3)))$ and $(\tilde{q},\tilde{r}) \neq (2,2(d-1)/(d-3)))$ when $d \geq 4$. 
Then, we have
\begin{align*}
	\norm{\int_{0}^{t} \cD(t-s) F(s) ds}_{L_{t}^{q} (I: L_{x}^{r}(\R^d))}
	&\cleq \norm{ \jbra{\nabla}^{\gamma+\tilde{\gamma}+\delta-1} F}_{L_{t}^{\tilde{q}'} (I: L_{x}^{\tilde{r}'}(\R^d))},
	\\
	\norm{\int_{0}^{t}  \partial_t \cD(t-s) F(s) ds}_{L_{t}^{q} (I: L_{x}^{r}(\R^d))}
	&\cleq \norm{ \jbra{\nabla}^{\gamma+\tilde{\gamma}+\delta} F}_{L_{t}^{\tilde{q}'} (I: L_{x}^{\tilde{r}'}(\R^d))},
\end{align*}
where $\delta = 0$ when $\frac{1}{\tilde{q}}(1/2-1/r)=\frac{1}{q}(1/2-1/\tilde{r})$ and in the other cases $\delta \geq 0$ is defined in the table 1 below. 
\begin{table}[htb]
\begingroup
\renewcommand{\arraystretch}{1.6}
\begin{tabular}{|c|c|c|} 
	\hline
	$\delta$
		&  $\frac{1}{\tilde{q}} \l( \frac{1}{2}- \frac{1}{r}\r) < \frac{1}{q} \l( \frac{1}{2} - \frac{1}{\tilde{r}} \r) $
		&  $\frac{1}{\tilde{q}} \l( \frac{1}{2}- \frac{1}{r}\r) > \frac{1}{q} \l( \frac{1}{2} - \frac{1}{\tilde{r}} \r)$ 
	\\[4pt]
	\hline \hline
	$\frac{d-1}{2}\l( \frac{1}{2}- \frac{1}{r}\r) \geq \frac{1}{q}$ 
	& $0$
	& $0$
	\\
	$\frac{d-1}{2}\l( \frac{1}{2}- \frac{1}{\tilde{r}}\r) \geq \frac{1}{\tilde{q}}$
	&  
	& 
	\\[5pt]
	\hline
	$\frac{d-1}{2}\l( \frac{1}{2}- \frac{1}{r}\r) \geq \frac{1}{q}$ 
	& $\times$
	& $\frac{\tilde{q}}{q}  \l\{ \frac{1}{\tilde{q}}-\frac{d-1}{2} \l( \frac{1}{2} - \frac{1}{\tilde{r}}\r) \r\}$
	\\
	$\frac{d-1}{2}\l( \frac{1}{2}- \frac{1}{\tilde{r}}\r) < \frac{1}{\tilde{q}}$
	&  
	& 
	\\[5pt]
	\hline
	$\frac{d-1}{2}\l( \frac{1}{2}- \frac{1}{r}\r) < \frac{1}{q}$ 
	& $\frac{q}{\tilde{q}}  \l\{ \frac{1}{q}-\frac{d-1}{2} \l( \frac{1}{2} - \frac{1}{r}\r) \r\}$
	& $\times$
	\\
	$\frac{d-1}{2}\l( \frac{1}{2}- \frac{1}{\tilde{r}}\r) \geq \frac{1}{\tilde{q}}$
	&  
	& 
	\\[5pt]
	\hline
	$\frac{d-1}{2}\l( \frac{1}{2}- \frac{1}{r}\r) < \frac{1}{q}$ 
	& $\frac{1}{\tilde{q}} \frac{d-1}{2} \l\{ \tilde{q}\l( \frac{1}{2}- \frac{1}{\tilde{r}}\r) - q\l( \frac{1}{2} - \frac{1}{r}\r) \r\}$
	& $\frac{1}{q} \frac{d-1}{2} \l\{ q\l( \frac{1}{2}- \frac{1}{r}\r) - \tilde{q}\l( \frac{1}{2} - \frac{1}{\tilde{r}}\r) \r\}$
	\\
	$\frac{d-1}{2}\l( \frac{1}{2}- \frac{1}{\tilde{r}}\r)< \frac{1}{\tilde{q}}$
	&
	&
	\\
	\hline
\end{tabular}
\caption{The value of $\delta$. ($\times$ means that the case does not occur.)} \label{tab1}
\endgroup
\end{table}
\end{proposition}

Roughly speaking, the Strichartz estimates of the damped wave equation is a combination of those of heat and wave. Indeed, the assumption of the exponents $(q,r)$ and $(\tilde{q},\tilde{r})$ comes from the low frequency part of the solution map $\cD(t)$, which behaves like the heat equation. On the other hand, the derivative losses
$\gamma$ and $\tilde{\gamma}$ come from the high frequency part, which behaves like the wave equation with exponential time decay. See the previous work \cite{Inup} for more detail.

It is worth remarking that we showed the homogeneous Strichartz estimate holds in the wave endpoint case, i.e., $(q,r)=(2,2(d-1)/(d-3)))$ when $d \geq4$. 
Its proof is based on the fact that the high frequency part of the homogeneous term can be reduced to
the wave solution map by the Mikhlin--H\"{o}rmander multiplier theorem and the endpoint estimate of the wave solution is obtained by Keel--Tao \cite{KeTa98}. On the other hand, we remove the wave endpoint case in the inhomogeneous Strichartz estimates (see Proposition \ref{prop1.2}). We can not use the Mihlin--H\"{o}rmander multiplier theorem to show the inhomogeneous Strichartz estimates since it has a time integral. That is why the wave endpoint inhomogeneous Strichartz estimates has not been known. 

In the present paper, we prove the wave endpoint inhomogeneous Strichartz estimates and we apply it to a nonlinear problem, especially unconditional uniqueness for the energy critical nonlinear damped wave equation.

\subsection{Main results}

First, we give the endpoint Strichartz estimate. 

\begin{theorem}[Endpoint Strichartz estimate]
\label{thm1.3}
Let $d \geq 4$. 
Then, we have
\begin{align}
\label{eq0.1}
	\norm{\int_{0}^{t} \cD(t-s) F(s) ds }_{L_{t}^{2} (I: L_{x}^{\frac{2(d-1)}{d-3}}(\R^d))}
	&\cleq \norm{ \jbra{\nabla}^{\frac{2}{d-1}} F}_{L_{t}^{2} (I: L_{x}^{\frac{2(d-1)}{d+1}}(\R^d))},
	\\
\notag
	\norm{ \int_{0}^{t}  \partial_t \cD(t-s) F(s) ds }_{L_{t}^{2} (I: L_{x}^{\frac{2(d-1)}{d-3}}(\R^d))}
	&\cleq \norm{ \jbra{\nabla}^{\frac{d+1}{d-1}} F}_{L_{t}^{2} (I: L_{x}^{\frac{2(d-1)}{d+1}}(\R^d))}.
\end{align}
\end{theorem}

The low frequency part of the inhomogeneous term can be treated easily since it behaves like the heat equation. We need to treat the high frequency part more carefully, which corresponds to the wave equation with exponential decay term related to time. We apply the argument of Keel--Tao \cite{KeTa98} to the high frequency part.  We need a small modification, since we  estimate $\int_{0}^{t} e^{-(t-s)/2} e^{i(t-s)|\nabla|} F(s)ds$. Due to its exponential decay, we get the bilinear Strichartz estimate for all $q,r \geq 2$. That is why we give a sketch of the proof for the reader's convenience. 

\begin{remark}
\label{rmk1.1.1}
We can easily show the following wave endpoint Strichartz estimates with additional derivative loss. 
\begin{align*}
	\norm{\int_{0}^{t} \cD(t-s) F(s) ds }_{L_{t}^{2} (I: L_{x}^{\frac{2(d-1)}{d-3}}(\R^d))}
	&\cleq \norm{ \jbra{\nabla}^{\frac{2}{d-1}+\varepsilon} F}_{L_{t}^{2} (I: L_{x}^{\frac{2(d-1)}{d+1}}(\R^d))},
\end{align*}
 for any $\varepsilon>0$. Indeed, modify the proof of Lemma 2.6 in \cite{Inup} as $(1+|t-s|N)^{-\frac{d-1}{2} \l( 1-\frac{2}{r}\r) +\varepsilon}$ instead of $(1+|t-s|N)^{-\frac{d-1}{2} \l( 1-\frac{2}{r}\r)}$. One of our contribution in this paper is taking $\varepsilon=0$. 
\end{remark}

Combining Theorem \ref{thm1.3} and the argument to prove Proposition \ref{prop1.2}, we obtain the following general Strichartz estimates containing the wave endpoint case. 

\begin{proposition}
\label{prop1.4}
Let $d\geq2$, $2\leq r,\tilde{r} < \infty$, and $2\leq q, \tilde{q} \leq \infty$. We set $\gamma:= \max \{ d(1/2-1/r)-1/q, \frac{d+1}{2}(1/2-1/r) \}$ and $\tilde{\gamma}$ in the same manner. Assume that $(q,r)$ and $(\tilde{q},\tilde{r})$ satisfies 
\begin{align*}
	\frac{d}{2}\l( \frac{1}{2} - \frac{1}{r}\r) +\frac{d}{2}\l( \frac{1}{2} - \frac{1}{\tilde{r}}\r)
	>\frac{1}{q} + \frac{1}{\tilde{q}},
\end{align*}
\begin{align*}
	\frac{d}{2} \l( \frac{1}{2} - \frac{1}{r}\r) + \frac{d}{2} \l( \frac{1}{2} - \frac{1}{\tilde{r}}\r) 
	= \frac{1}{q} + \frac{1}{\tilde{q}}
	\text{ and }
	1< \tilde{q}' < q<\infty,
\end{align*}
or 
\begin{align*}
	(q,r)=(\tilde{q},\tilde{r})=(\infty,2).
\end{align*}
(We may take $(q,r) = (2,2(d-1)/(d-3))$ or  $(\tilde{q},\tilde{r}) = (2,2(d-1)/(d-3))$ when $d \geq 4$. )
Then, we have
\begin{align*}
	\norm{ \int_{0}^{t} \cD(t-s) F(s) ds }_{L_{t}^{q} (I: L_{x}^{r}(\R^d))}
	&\cleq \norm{ \jbra{\nabla}^{\gamma+\tilde{\gamma}+\delta-1} F }_{L_{t}^{\tilde{q}'} (I: L_{x}^{\tilde{r}'}(\R^d))},
	\\
	\norm{\int_{0}^{t}  \partial_t \cD(t-s) F(s) ds }_{L_{t}^{q} (I: L_{x}^{r}(\R^d))}
	&\cleq \norm{ \jbra{\nabla}^{\gamma+\tilde{\gamma}+\delta} F}_{L_{t}^{\tilde{q}'} (I: L_{x}^{\tilde{r}'}(\R^d))},
\end{align*}
where $\delta = 0$ when $\frac{1}{\tilde{q}}(1/2-1/r)=\frac{1}{q}(1/2-1/\tilde{r})$ and in the other cases $\delta \geq 0$ is defined in the table 1 above.
\end{proposition}

We also get the following Besov type estimate from Proposition \ref{prop1.1} and \ref{prop1.4} by the Littlewood--Paley decomposition. These Besov type Strichartz estimates are useful to analyze nonlinear problems. 

\begin{proposition}[Besov type homogeneous Strichartz estimates]
\label{prop1.5}
Let $s \in \R$. Assume that $(q,r)$ satisfies the assumptions in Proposition \ref{prop1.1}. Let $\gamma$ be as in Proposition \ref{prop1.1}. Then we have the following Besov type homogeneous Strichartz estimates.
\begin{align*}
	\norm{\cD(t) f}_{L^q(I:B^{s}_{r,2}(\R^d))} 
	&\cleq \norm{\jbra{\nabla}^{\gamma-1} f }_{B^{s}_{2,2}},
	\\
	\norm{\partial_t \cD(t)f}_{L_{t}^{q} (I: B^{s}_{r,2}(\R^d))} 
	&\cleq \norm{\jbra{\nabla}^{\gamma} f }_{B^{s}_{2,2}},
	\\
	\norm{\partial_t^2 \cD(t)f}_{L_{t}^{q} (I: B^{s}_{r,2}(\R^d))} 
	&\cleq \norm{\jbra{\nabla}^{\gamma+1} f }_{B^{s}_{2,2}},
\end{align*}
where  $I \subset \R$ is a time interval.
\end{proposition}

\begin{proposition}[Besov type inhomogeneous Strichartz estimates]
\label{prop1.6}
Let $s \in \R$. Assume that $(q,r)$ and $(\tilde{q},\tilde{r})$ satisfy the assumptions in Proposition \ref{prop1.4}. Let $\gamma$ and $\tilde{\gamma}$ be as in Proposition \ref{prop1.4} and $\delta$ be defined in Table \ref{tab1}. Then we have the following Besov type inhomogeneous Strichartz estimates.
\begin{align*}
	\norm{ \int_{0}^{t} \cD(t-s) F(s) ds}_{L^q(I:B^{s}_{r,2}(\R^d))} 
	&\cleq \norm{\jbra{\nabla}^{\gamma+\tilde{\gamma} + \delta-1} F }_{L^{\tilde{q}'}(I:B^{s}_{\tilde{r}',2}(\R^d))},
	\\
	\norm{ \int_{0}^{t} \partial_t \cD(t-s) F(s) ds}_{L^q(I:B^{s}_{r,2}(\R^d))} 
	&\cleq \norm{\jbra{\nabla}^{\gamma+\tilde{\gamma} + \delta} F }_{L^{\tilde{q}'}(I:B^{s}_{\tilde{r}',2}(\R^d))},
\end{align*}
where $I \subset \R$ is a time interval.
\end{proposition}

\begin{remark}
In fact, we can take $r=\infty$ in the Besov type inequality. 
\end{remark}

In the present paper, we also discuss the application of the endpoint Strichartz estimates to a nonlinear problem.
We consider the following energy critical nonlinear damped wave equation.
\begin{align}
\label{NLDW}
\tag{NLDW}
\left\{
\begin{array}{ll}
	\partial_t^2 u - \Delta u +\partial_t u = |u|^{\frac{4}{d-2}}u, & (t,x) \in [0,T) \times \mathbb{R}^d,
	\\
	(u(0),\partial_t u(0)) = (u_0,u_1), & x\in \mathbb{R}^d,
\end{array}
\right.
\end{align}
where $d \geq 3$, $(u_0,u_1)$ is given, and $u$ is an unknown complex valued function. In the previous paper \cite{Inup} (see also \cite{Wat17}), we show the local well-posedness of \eqref{NLDW} when $3 \leq d \leq 5$ and we give the propositions of the behavior of the solutions for $d \geq 3$. In this paper, we will show the  local well-posedness of \eqref{NLDW} when $d \geq 6$. 
In these local well-posedness, we needed auxiliary function spaces to prove the uniqueness. We will show the unconditional uniqueness of the solution to \eqref{NLDW} in the energy space $H^1(\R^d) \times L^2(\R^d)$ when $d \geq 4$.
Namely, we will remove those auxiliary spaces by applying the endpoint Strichartz estimates. We give the precise definition of the solution to \eqref{NLDW}.

\begin{definition}[solution]
\label{def1.1}
We say that $u$ is a solution to \eqref{NLDW} on a time interval $I$ with $0 \in I$ if $u$ satisfies $(u,\partial_t u) \in C(I:H^1(\R^d)\times L^2(\R^d))$ 
and the Duhamel formula
\begin{align*}
	u(t,x)=\cD(t) (u_0+u_1) +\partial_t \cD(t) u_0 + \int_{0}^{t} \cD(t-s) ( |u(s)|^{\frac{4}{d-2}}u(s) ) ds
\end{align*}
in the sense of tempered distributions for every $t \in I$.
\end{definition}

\begin{remark}
We emphasize that the solutions may not belong to the Strichartz spaces. This definition is different from Definition 1.1 in \cite{Inup}.
\end{remark}

First, we show the local well-posedness when $d \geq 6$, which was not treated in \cite{Inup}.

\begin{theorem}[L.W.P when $d \geq 6$]\label{thm_lwp}
Let $d \geq 6$ and $T\in (0,\infty]$. Let $(u_0,u_1) \in H^1(\R^d)\times L^2(\R^d)$ satisfy $\| (u_0,u_1) \|_{H^1 \times L^2} \leq A$. Then, there exists $\delta=\delta(A)>0$ such that if 
\begin{align*}
	\norm{\cD(t) (u_0+u_1) +\partial_t \cD(t) u_0}_{L_{t,x}^{\frac{2(d+1)}{d-2}}([0,T))} \leq \delta,
\end{align*}
then there exists a solution $u$ to \eqref{NLDW} with $\|  u \|_{L_{t,x}^{2(d+1)/(d-2)}}([0,T)) \leq 2 \delta$. Moreover, we have the standard blow-up criterion, that is, if the maximal existence time $T_{+}=T_{+}(u_0,u_1)$ is finite, then the solution satisfies $\norm{u}_{L^{2(d+1)/(d-2)}([0,T_{+}))}=\infty$. 
\end{theorem}

The local well-posedness can be proved more easily. See e.g. \cite{IkWa17p}. The argument is based on the perturbation argument of wave or Klein-Gordon equations. However, their method only works for the local problem. On the other hand, our method based on the Strichartz estimates for the damped wave equation can treat the global property. It is worth emphasizing that Theorem \ref{thm_lwp} implies small data global existence, which is one of our important contribution. This also implies the decay of small solutions (see \cite[Theorem 1.4]{Inup}) Moreover, we can show long time perturbation by the similar argument to the proof of Theorem \ref{thm_lwp}.

We  give the unconditional uniqueness of the solution to \eqref{NLDW} in the energy space $H^1(\R^d) \times L^2(\R^d)$ for $d \geq 4$. 

\begin{theorem}[unconditional uniqueness]\label{thm_uu}
Let $d \geq 4$. Let $u,v$ be solutions (in the sense of Definition \ref{def1.1}) to \eqref{NLDW} with the initial data $(u_0,u_1), (v_0,v_1) \in H^1(\R^d) \times L^2(\R^d)$, respectively. If $(u_0,u_1)=(v_0,v_1)$, then we have $u=v$. 
\end{theorem}

\begin{remark}
The unconditional uniqueness is a local problem. However, it seems to be difficult to apply the result for the corresponding energy critical nonlinear wave equation to our unconditional uniqueness problem since we need to treat the inhomogeneous Sobolev space. If we have the result for the energy critical nonlinear Klein-Gordon equation, we can solve our problem as a perturbation of it. However, the authors could not find the result for the Klein-Gordon equation.  Our method in this paper may be applicable to the Klein-Gordon equation. 
\end{remark}

\begin{remark}
The unconditional uniqueness when $d=3$ remains open since its endpoint homogeneous Strichartz estimate does not hold. See Appendix D. 
\end{remark}

For the corresponding energy critical nonlinear wave equation, the local well-posedness when $d \geq6$ and the unconditional uniqueness when $d \geq4$ was obtained by Bulut et.al. \cite{Bu13}. The proof of the local well-posedness when $d \geq6$ depends on exotic Strichartz estimates and the proof of the unconditional uniqueness relies on paraproduct estimates for the homogeneous Besov spaces. Though our argument is based on their argument, the homogeneous Besov spaces do not match the damped wave equation. We need the inhomogeneous Sobolev and Besov spaces and thus, in particular, we need to show paraproduct estimates for the inhomogeneous Besov spaces.

\begin{notation}
For the exponent $p$, we denote the H\"{o}lder conjugate of $p$ by $p'$. The bracket $\jbra{\cdot}$ is Japanese bracket, i.e., $\jbra{a}:=(1+|a|^2)^{1/2}$. 

We use $A \lesssim B$ to denote the estimate $A \leq CB$ with some constant $C>0$.
The notation $A \sim B$ stands for $A \lesssim B$ and $B \lesssim A$.

For a function $f : \mathbb{R}^n \to \mathbb{C}$,
we define the Fourier transform and the inverse Fourier transform by
\begin{align*}
	\mathcal{F} [f] (\xi) = \hat{f} (\xi) = (2\pi)^{-n/2}
		\int_{\mathbb{R}^d} e^{-i x \xi} f(x) dx, \\ 
	\mathcal{F}^{-1} [f] (x) = (2\pi)^{-n/2}
		\int_{\mathbb{R}^d} e^{i x \xi} f(\xi) d\xi.
\end{align*}

For a measurable function $m = m(\xi)$, we denote
the Fourier multiplier $m(\nabla)$ by
\begin{align*}
	m(\nabla) f (x) = \cF^{-1} \left[ m(\xi) \hat{f}(\xi) \right] (x).
\end{align*}

For $s \in \mathbb{R}$ and $1\leq p \leq \infty$,
we denote the usual Sobolev space by
\begin{align*}
	W^{s,p}(\mathbb{R}^d) := \left\{ f \in \mathcal{S}' (\mathbb{R}^d) :
		\| f \|_{W^{s,p}} = \| \langle \nabla \rangle^s f \|_{L^p} < \infty \right\}.
\end{align*}
We write $H^s(\mathbb{R}^d) := W^{s,2}(\mathbb{R}^d)$ for simplicity. Let $\dot{W}^{s,p}(\mathbb{R}^d)$ and $\dot{H}^{s}(\R^d)$ denote the corresponding homogeneous Sobolev spaces.

For a time interval $I$ and $F:I\times \R^d \to \C$, we set 
\begin{align*}
	\norm{F}_{L^{q}(I:L^r(\R^d))}:=  \l(\int_{I} \norm{F(t,\cdot)}_{L^r(\R^d)}^q dt \r)^{1/q}
\end{align*}
and $L_{t,x}^{q}(I):=L^{q}(I:L^{q}(\R^d))$. We sometimes use $L_s^p$ and $L_t^p$ to uncover time variables $s$ and $t$. 

We define the Littlewood--Paley decomposition as follows.
Let $\chi$ be a radial nonnegative $C^{\infty}$-function supported in
$\{ \xi \in \mathbb{R}^d: |\xi| \leq \frac{25}{24}\}$ with $\chi(\xi)=1$
on the unit ball $\{ \xi \in \mathbb{R}^d: |\xi| \leq 1 \}$. For $a>0$, we define
\begin{align*}
	\chi_{\leq a} (\xi) := \chi \left( \frac{\xi}{a} \right)
	\text{ and }
	\chi_{> a} (\xi) :=1- \chi_{\leq a} (\xi).
\end{align*}
For each integer $j \in \mathbb{Z}$, we set 
\begin{align*}
	\widehat{ \Delta_{\leq j} f} (\xi) 
	&:= \chi_{\leq 2^{j}} (\xi) \widehat{f}(\xi),
	\\
	\widehat{ \Delta_{> j} f} (\xi) 
	&:= \chi_{> 2^{j}} (\xi) \widehat{f}(\xi),
	\\
	\widehat{ \Delta_{j} f} (\xi) 
	&=  \chi_{ 2^{j}} (\xi) \widehat{f}(\xi)
	:=\left( \chi_{\leq 2^{j}} (\xi) - \chi_{\leq 2^{j-1}} (\xi) \right)\widehat{f}(\xi).
\end{align*}
We also define
\begin{align*}
	\widehat{ \Delta_{< j} f} (\xi) 
	&:=\chi_{\leq 2^{j-1}} \widehat{f}(\xi),
	\\
	\widehat{ \Delta_{\geq j} f} (\xi) 
	&:= \chi_{> 2^{j-1}} (\xi) \widehat{f}(\xi),
	\\
	\widehat{ \Delta_{j < \cdot \leq l } f} (\xi) 
	&:= \widehat{ \Delta_{\leq l } f} (\xi) - \widehat{ \Delta_{\leq j } f} (\xi) 
	= \left( \chi_{\leq 2^{l}} (\xi) - \chi_{\leq 2^{j}} (\xi) \right)\widehat{f}(\xi)
\end{align*}
for $j <l$. 
Moreover we assume that $\{\chi_j\}_{j \in \mathbb{Z}}$ give a dyadic partition of unity as follows. 
\begin{align*}
	\Delta_{\leq 0} + \sum_{j=1}^{\infty} \Delta_j =\mathrm{Id} \text{ on } \mathscr{S}'
	\text{ and } 
	\sum_{j\in \mathbb{Z}} \Delta_j =\mathrm{Id} \text{ on } \mathscr{S}'.
\end{align*}
We also set $\tilde{\Delta}_{j}=\Delta_{j-1} + \Delta_{j}+\Delta_{j+1}$
and correspondingly, $\widetilde \chi_{2^j}(\xi) = \chi_{2^{j-1}}(\xi) + \chi_{2^j}(\xi) + \chi_{2^{j+1}}(\xi)$.

For $1\leq p,q \leq \infty$ and $s \in \mathbb{R}$, we define inhomogeneous Besov norm by
\begin{align*}
	\| f \|_{B_{p,q}^{s}} 
	:=\| \Delta_{\leq 0} f \|_{L^p} + \left\|  \{ 2^{js} \| \Delta_j f \|_{L^p}  \}_{j=1}^{\infty} \right\|_{l^{q}}
\end{align*}
and inhomogeneous Besov space by
\begin{align*}
	B_{p,q}^{s}(\R^d) := \{ f \in \mathscr{S}' : \| f \|_{B_{p,q}^{s}} < \infty \}.
\end{align*}
We denote the homogeneous space of a space-time function space $\mathcal{X}$ by $\dot{\mathcal{X}}$, that is, we replace the inhomogeneous Sobolev space or the inhomogeneous Besov spaces by their homogeneous spaces and we do not change the Lebesgue space. 

\end{notation}

This paper is structured as follows. 
Section \ref{sec2} is devoted to show the endpoint Strichartz estimates. Especially, the $L^\infty$-$L^1$ estimate is proven in Section \ref{sec2.1} and the bilinear estimates and the endpoint Strichartz estimate are given in Section \ref{sec2.2}. In Section \ref{sec3}, we prove the local well-posedness of \eqref{NLDW} when $d\geq6$.  In Section \ref{sec4}, we show the unconditional uniqueness of \eqref{NLDW}.  In Appendix, we collect some lemmas and we give the proof of the Besov type Strichartz estimates. In Appendix D, we show the endpoint homogeneou Strichartz estimate does not hold when $d=3$. 


\section{Endpoint Strichartz estimate}
\label{sec2}

We have the following lemma for the low frequency part. 
\begin{lemma}[{\cite[Lemma 2.3]{Inup}}]\label{lem_str_est_low}
Let $1 \leq \tilde{r}' \leq r \leq \infty$ and  $1\leq q, \tilde{q} \leq \infty$. 
Assume that they satisfy 
\begin{align*}
	\frac{d}{2} \l( \frac{1}{2} - \frac{1}{r}\r) + \frac{d}{2} \l( \frac{1}{2} - \frac{1}{\tilde{r}}\r)
	> \frac{1}{q} + \frac{1}{\tilde{q}},
\end{align*}
or
\begin{align*}
	\frac{d}{2} \l( \frac{1}{2} - \frac{1}{r}\r) + \frac{d}{2} \l( \frac{1}{2} - \frac{1}{\tilde{r}}\r) 
	= \frac{1}{q} + \frac{1}{\tilde{q}}
	\text{ and }
	1< \tilde{q}' < q<\infty.
\end{align*}
Then it holds that 
\begin{align*}
	\norm{ \int_{0}^{t} \cD(t-s) P_{\leq 1} F(s) ds}_{L^{q}(I:L^r(\R^d))}
	&\cleq \norm{F}_{L^{\tilde{q}'}(I:L^{\tilde{r}'}(\R^d))},
	\\
	\norm{\int_{0}^{t} \partial_t \cD(t-s) P_{\leq 1} F(s) ds}_{L^{q}(I:L^r(\R^d))}
	&\cleq \norm{F}_{L^{\tilde{q}'}(I:L^{\tilde{r}'}(\R^d))},
\end{align*}
where $I \subset [0,\infty)$ is a time interval such that $0 \in \overline{I}$ and the implicit constant is independent of $I$. 
\end{lemma}

Therefore, the endpoint case holds for the low frequency part. It is enough to consider the high frequency part. 
We apply the method of Keel and Tao \cite{KeTa98} to prove the endpoint Strichartz estimate for the high frequency part. In this section we use $N \in 2^{\Z}$ to denote the dyadic numbers for spatial variables in order to distinguish dyadic numbers for time variable, which are denoted by $2^l$, and them. We set $P_N = \Delta_{j}$ for $N=2^j$. $P_{>N}$, $P_{\leq N}$ etc. are defined in the same way.

\subsection{Stationary phase method}
\label{sec2.1}

First, we give an $L^{\infty}$-$L^{1}$ estimate for the high frequency part of the linear solution.

\begin{lemma}
The following estimate is valid for $t>0$. 
\begin{align}
\label{eq2.1}
	\norm{ e^{it\sqrt{-\Delta-1/4}} P_{>1} P_{N}  f }_{ L^{\infty} }
	\cleq ( 1 + t N )^{ -\frac{d-1}{2} } N^{d} \norm{ P_{N} f }_{ L^{1} }.
\end{align}
\end{lemma}

\begin{proof}
The proof is very similar to the $L^{\infty}$-$L^{1}$ estimate for the wave equation. See e.g. \cite[Lemma 2.1]{KTV14}. However, we give a proof for the reader's convenience. 
We only treat the case of $N \geq 1$ since $P_{N}P_{>1}=0$ when $N \leq 1/2$. 
Since $P_{N} \widetilde{P}_{N} = 1$, it follows that
\begin{align*}
	e^{it\sqrt{-\Delta-1/4}} P_{>1} P_{N}  f 
	= \{  e^{it\sqrt{|\cdot|^2-1/4}} \chi_{>1} \widetilde{\chi}_{N}  \}^{\vee} * (P_{N} f).
\end{align*}
It is enough to show 
\begin{align*}
	\l| \int_{\R^d} e^{ix\cdot \xi} e^{it\sqrt{|\xi|^2-1/4}} \chi_{>1}(\xi) \widetilde{\chi}_{N}(\xi) d\xi \r|
	\cleq ( 1 + t N )^{ -\frac{d-1}{2} } N^{d}
\end{align*}
by the Young inequality. 
We set 
\begin{align}
	\notag
	I 
	&:=  \int_{\R^d} e^{ix\cdot \xi} e^{it\sqrt{|\xi|^{2}-1/4}} \chi_{>1}(\xi) \widetilde{\chi}_{N}(\xi) d\xi
	\\
	\label{eq2.2}
	&= N^{d} \int_{\S^{d-1}} \int_{0}^{\infty}   e^{ it\sqrt{(Nr)^{2}-1/4}+iNxr\omega } \chi_{>1}(Nr) \widetilde{\chi}_{1} (r) r^{d-1} dr d\sigma(\omega)
	\\
	\label{eq2.3}
	&=N^{d} \int_{0}^{\infty}   e^{ it\sqrt{(Nr)^{2}-1/4} }
		\check{\sigma} (Nrx) \chi_{>1}(Nr) \widetilde{\chi}_{1}(r) r^{d-1} dr,
\end{align}
where
$\check{\sigma}(x)=(2\pi)^{-\frac{d}{2}} \int_{\S^{d-1}} e^{i|x|\omega_d} d\sigma(\omega)$.

{\bf Case 1.} First, we assume ``$d=1$" or ``$d \geq 2$ and $t \cleq N^{-1}$". By \eqref{eq2.2}, we have
\begin{align*}
	|I| \leq N^d
\end{align*}
When ``$d=1$" or ``$d \geq 2$ and $|t| \leq N^{-1}$", we have $( 1 + t N )^{ -\frac{d-1}{2} } \cgeq 1$. Thus, we obtain the desired estimate.

{\bf Case 2.} We assume that $d \geq 2$ and $t \gg N^{-1}$. 

{\bf Case 2-1.} We consider the case of $|x| \cgeq t$. Since $|\check{\sigma}(x)| \cleq \jbra{x}^{-\frac{d-1}{2}}$ (see \cite[Lemma 2.2]{KTV14}), it follows from \eqref{eq2.3}, $|x| \cgeq t$, $\chi_{>1} \leq 1$, and $t \gg N^{-1}$ that  
\begin{align*}
	|I|
	&\cleq N^{d} \int_{0}^{\infty}   
		\jbra{Nr|x|}^{-\frac{d-1}{2}} \chi_{>1}(Nr) \widetilde{\chi}_{1}(r) r^{d-1} dr
	\\
	&\cleq N^{\frac{d+1}{2}} |x|^{-\frac{d-1}{2}} \int_{0}^{\infty} \chi_{>1}(Nr) \widetilde{\chi}_{1}(r) r^{d-1} dr
	\\
	&\cleq N^{\frac{d+1}{2}} t^{-\frac{d-1}{2}} 
	\\
	&\cleq ( 1 + t N )^{ -\frac{d-1}{2} } N^{d}.
\end{align*}

{\bf Case 2-2.} We consider $|x| \ll t$. We use the formula in \eqref{eq2.2}. Let $\phi(r):= t\sqrt{(Nr)^{2}-1/4} + Nxr\omega$. Then, we have $\phi'(r)=tN^{2}r / \sqrt{(Nr)^{2}-1/4} + Nx\omega$ and thus $|\phi'(r)| \cgeq Nt $ since $t \gg |x|$. When $N \geq 2^{2}$, by $ \chi_{>1}(Nr) \widetilde{\chi}_{1} (r)= \widetilde{\chi}_{1} (r)$, $e^{i\phi(r)}= (i\phi'(r))^{-1} \partial_r e^{i\phi(r)}$ and the integration by parts $k$ times, we get
\begin{align*}
	|I|
	\cleq N^{d} (Nt)^{-k}
	\cleq N^{\frac{d+1}{2}} t^{-\frac{d-1}{2}},
\end{align*}
where we take $k=(d-1)/2$ if $d$ is odd and $k=d/2$ if $d$ is even and we use $t \gg N^{-1}$. When $N=1,2$, by the integration by parts $k$ times, in the same way as above, we obtain
\begin{align*}
	|I|
	\leq C_{N} t^{-\frac{d-1}{2}}
	\leq \max_{ N=1,2 }\{C_{N}\}  t^{-\frac{d-1}{2}}.
\end{align*}
Therefore, for $N \geq 1$, we get
\begin{align*}
	|I|
	\cleq N^{\frac{d+1}{2}} t^{-\frac{d-1}{2}}
	\cleq ( 1 + t N )^{ -\frac{d-1}{2} } N^{d}
\end{align*}
since $t \gg N^{-1}$. This completes the proof. 
\end{proof}

Combining \eqref{eq2.1} and 
\begin{align*}
	\norm{ e^{it\sqrt{-\Delta-1/4}} P_{>1} P_{N}  f }_{ L^{2} }
	=  \norm{ P_{N} f }_{ L^{2} },
\end{align*}
we get the following $L^{r}$-$L^{r'}$ estimate by the interpolation.

\begin{corollary}
\label{cor2.2}
Let $2 \leq r \leq \infty$. We have
\begin{align*}
	\norm{ e^{it\sqrt{-\Delta-1/4}} P_{>1} P_{N}  f }_{ L^{r} }
	\cleq ( 1 + t N )^{ -\frac{(d-1)(r-2)}{2r} } N^{ \frac{d(r-2)}{r} } \norm{ P_{N} f }_{ L^{r'} },
\end{align*}
for $t>0$.
\end{corollary}


\subsection{Bilinear estimate and the proof of the endpoint Strichartz estimates}
\label{sec2.2}

We estimate the space-time norm of the inhomogeneous term:
\begin{align*}
	J:=\norm{ \int_{0}^{\infty} e^{-\frac{t-s}{2}} e^{i(t-s)\sqrt{-\Delta-1/4}} P_{>1} P_{N} F(s) ds }_{L^{q}(I:L^{r}(\R^d))}.
\end{align*}
If $N\leq 1/2$, $J=0$ since $P_{>1} P_{N}=0$. Thus, we only consider the case of $N \geq 1$. Here, we only treat the case of $N \geq 2$. It follows from the duality that 
\begin{align*}
	J
	&= \norm{ \int_{0}^{\infty} e^{-\frac{t-s}{2}} e^{i(t-s)\sqrt{-\Delta-1/4}} P_{>1} \widetilde{P}_{N} P_{N} F(s) ds }_{L^{q}(I:L^{r}(\R^d))}
	\\
	&= \sup_{ \norm{G}_{ L_{t}^{q'} L_{x}^{r'} } = 1 }
	\l| \int_{0}^{\infty} \int_{\R^d} \int_{0}^{\tau}  e^{-\frac{\tau-s}{2}} e^{i(\tau-s)\sqrt{-\Delta-1/4}} P_{>1} \widetilde{P}_{N} P_{N} F(s,x) \overline{G(\tau,x)} ds dx d\tau \r|
	\\
	&= \sup_{ \norm{G}_{ L_{t}^{q'} L_{x}^{r'} } = 1 } |T(F,G)|,
\end{align*}
where we set
\begin{align*}
	T(F,G) 
	&:= \int_{0}^{\infty} \int_{0}^{\tau}  e^{-\frac{\tau-s}{2}}
	\tbra{ e^{-is\sqrt{-\Delta-1/4}} P_{N} F(s) }{  e^{-i\tau \sqrt{-\Delta-1/4}} P_{>1} \widetilde{P}_{N}  G(\tau) }_{ L_{x}^{2} }
	ds  d\tau.
\end{align*}
Set $\cW(t):=e^{it\sqrt{-\Delta-1/4}} P_{>1}$ for simplicity. Since we have $P_{N}=P_{>1}P_{N}$ from $N \geq 2$, it follows that
\begin{align*}
	T(F,G) 
	= \int_{0}^{\infty} \int_{0}^{\tau}  e^{-\frac{\tau-s}{2}}
	\tbra{ \cW^{*} (s)P_{N} F(s) }{  \cW^{*} (\tau) \widetilde{P}_{N}  G(\tau) }_{ L_{x}^{2} }
	ds  d\tau,
\end{align*}
where $\cW^{*}$ denotes the conjugate of $\cW$. We set 
\begin{align*}
	I(\tau,s):= e^{-\frac{\tau-s}{2}}
	\tbra{ \cW^{*} (s)P_{N} F(s) }{  \cW^{*} (\tau) \widetilde{P}_{N}  G(\tau) }_{ L_{x}^{2} },
\end{align*}
for $\tau \geq s$. 

\begin{lemma}
\label{lem2.3}
For $r \in [2,\infty]$, we have
\begin{align*}
	| I(\tau,s) |
	\cleq e^{ -\frac{\tau-s}{2} }  \{1+(\tau-s)N\}^{ -\frac{(d-1)(r-2)}{2r} }   N^{ \frac{d(r-2)}{r} }  \norm{ P_{N}F(s) }_{L^{r'}} \norm{ \widetilde{P}_{N}G(\tau) }_{L^{r'}},
\end{align*}
for $\tau \geq s$.
\end{lemma}

\begin{proof}
By the H\"{o}lder inequality and Corollary \ref{cor2.2}, we have
\begin{align*}
	| I(\tau,s) |
	&\cleq e^{ -\frac{\tau-s}{2} }   \norm{ \cW (\tau-s) P_{N} F(s) }_{L^{r}} \norm{ \widetilde{P}_{N}  G(\tau) }_{L^{r'}}
	\\ \notag
	&\cleq e^{ -\frac{\tau-s}{2} }  
	\{ 1 + (\tau-s) N \}^{ -\frac{(d-1)(r-2)}{2r} } N^{ \frac{d(r-2)}{r} } \norm{ P_{N} F(s)  }_{ L^{r'} } \norm{ \widetilde{P}_{N}  G(\tau) }_{L^{r'}},
\end{align*}
for $\tau \geq s$. 
\end{proof}

From this, we get the following bilinear estimate for non-endpoint admissible pairs. 

\begin{lemma}\label{lem_bilin_est_nonadm}
Let $q \in [2,\infty]$, $r \in [2,\infty]$, and $(q,r) \neq (2,2(d-1)/(d-3))$. 
We have
\begin{align*}
	|T(F,G)|
	\cleq N^{2\gamma} \norm{ P_{N}  F }_{L^{q'}(I:L^{r'}(\R^d))} \norm{ \widetilde{P}_{N}  G }_{L^{q'}(I:L^{r'}(\R^d))}
\end{align*}
\end{lemma}

\begin{proof}
From Lemma \ref{lem2.3} and the H\"{o}lder inequality for $\tau$, we obtain
\begin{align*}
	&|T(F,G)|
	\\
	&\leq  \int_{0}^{\infty} \int_{0}^{\tau}  |I(\tau,s)| ds d\tau
	\\
	&\cleq   \int_{0}^{\infty} \int_{0}^{\tau}  e^{ -\frac{\tau-s}{2} }  \{1+(\tau-s)N\}^{ -\frac{(d-1)(r-2)}{2r} }   N^{ \frac{d(r-2)}{r} }  \norm{ P_{N}F(s) }_{L^{r'}} \norm{ \widetilde{P}_{N}G(\tau) }_{L^{r'}}  ds d\tau
	\\
	&\cleq  N^{ \frac{d(r-2)}{r} }  \norm{ \int_{0}^{\tau}  e^{ -\frac{\tau-s}{2} }  \{1+(\tau-s)N\}^{ -\frac{(d-1)(r-2)}{2r} } \norm{ P_{N}F(s) }_{L^{r'}}  ds }_{L^{q}}
	 \norm{ \widetilde{P}_{N}G}_{L^{q'}L^{r'}}
	\\
	&\cleq N^{2\gamma} \norm{ P_{N} F }_{L_{t}^{q'}L_{x}^{r'}} \norm{ \widetilde{P}_{N} G }_{L_{t}^{q'}L_{x}^{r'}},
\end{align*}
where we used the Young inequality or the Hardy--Littlewood--Sobolev inequality in the last. See \cite[Lemma 2.6]{Inup} for the detail. 
\end{proof}

For $l \in \Z$, we set 
\begin{align*}
	T_{l}(F,G) 
	&:= \int_{0}^{\infty} \int_{[0,\tau] \cap [\tau-2^{l+1},\tau-2^{l}]} I(\tau,s) ds  d\tau.
\end{align*}
Then, we have
\begin{align*}
	\sum_{l \in \Z} T_{l}(F,G)  = T(F,G).
\end{align*}
For simplicity, we set $K_{l}=K_{l}(\tau):=[0,\tau] \cap [\tau-2^{l+1},\tau-2^{l}]$. 
Let $d \geq 4$ and $(q^{*},r^{*}) = (2,2(d-1)/(d-3))$. 

\begin{lemma}
\label{lem2.5}
Let $2 \leq a , b \leq \infty$ and $\beta(a,b):= -1 + \frac{d-1}{2}(1-\frac{1}{a} -\frac{1}{b})$. For $(1/a,1/b)$ near $(1/r^{*},1/r^{*})$, the following is valid.
\begin{align*}
	|T_{l}(F,G)| \cleq 2^{-l\beta(a,b)} N^{\gamma(a)+\gamma(b)} \norm{ P_{N} F }_{L_{t}^{2}L_{x}^{a'}} \norm{ \widetilde{P}_{N}  G }_{L_{t}^{2}L_{x}^{b'}},
\end{align*}
where $\gamma(a)=d(1/2-1/a)-1/q(a)$ and $1/q(a)=\frac{d-1}{2}(1/2-1/a)$.
\end{lemma}

\begin{proof}
By the interpolation, it is enough to show the inequality  in the the following cases.
\begin{enumerate}
\item $a=b=\infty$.
\item $2 \leq a <r^{*}$ and $b=2$.
\item $2 \leq b < r^{*}$ and $a=2$. 
\end{enumerate}

(1). By Lemma \ref{lem2.3} as $r=\infty$, $\tau -s \ceq 2^{l}$ in $K_{l}$ and the Cauchy--Schwarz inequality, we get
\begin{align*}
	|T_{l}(F,G)|
	&\leq \int_{0}^{\infty} \int_{ K_{l} } |I(\tau,s)| ds  d\tau
	\\
	&\cleq \int_{0}^{\infty} \int_{ K_{l} }  \{ (\tau-s) N \}^{ -\frac{d-1}{2} } N^{d}  \norm{ P_{N} F(s) }_{L^{1}} \norm{ \widetilde{P}_{N}  G(\tau) }_{L^{1}} ds  d\tau
	\\
	&\cleq  2^{ -\frac{d-1}{2}l } N^{ \frac{d+1}{2} }  \int_{0}^{\infty} \int_{ K_{l} }  \norm{ P_{N} F(s) }_{L^{1}} ds \norm{ \widetilde{P}_{N}  G(\tau) }_{L^{1}}  d\tau
	\\
	&\cleq  2^{ -\frac{d-1}{2}l } N^{ 2\gamma(\infty)}   \norm{ H }_{L^{2}}  \norm{ \widetilde{P}_{N}  G }_{L_{t}^{2}L_{x}^{1}},
\end{align*}
where we set 
\begin{align*}
	H(\tau) :=  \int_{ K_{l}(\tau) }  \norm{ P_{N} F(s) }_{L^{1}} ds.
\end{align*}
Here, we have
\begin{align*}
	\norm{H}_{L^{2}}^{2} 
	&=  \int_{0}^{\infty} \l( \int_{ K_{l}(\tau) }  \norm{ P_{N} F(s) }_{L^{1}} ds \r)^{2} d\tau
	\\
	&\cleq 2^{l}  \int_{0}^{\infty}  \int_{ K_{l}(\tau) }  \norm{ P_{N} F(s) }_{L^{1}}^{2} ds  d\tau
	\\
	&= 2^{l}  \int_{0}^{\infty}  \int_{ [s,\infty] \cap [s+2^{l},s+2^{l+1}] } d\tau \norm{ P_{N} F(s) }_{L^{1}}^{2}  ds 
	\\
	&\cleq 2^{2l} \norm{ P_{N} F(s) }_{L_{t}^{2}L_{x}^{1}}^{2}.
\end{align*}
Therefore, we obtain 
\begin{align*}
	|T_{l}(F,G)|
	\cleq 2^{\l( -\frac{d-1}{2} +1\r)l} N^{2\gamma(\infty)} \norm{ P_{N} F(s) }_{L_{t}^{2}L_{x}^{1}} \norm{ \widetilde{P}_{N}  G }_{L_{t}^{2}L_{x}^{1}}.
\end{align*}

(2). We have
\begin{align*}
	&|T_{l}(F,G)|
	\\
	&= \l|  \int_{0}^{\infty} \int_{0}^{\infty} \1_{[\tau-2^{l+1},\tau-2^{l}]}(s) I(\tau,s) ds  d\tau \r|
	\\
	&= \l| \int_{0}^{\infty} \int_{0}^{\infty}
	e^{-\frac{\tau-s}{2}}
	\tbra{ \cW^{*} (s)  \1_{[\tau-2^{l+1},\tau-2^{l}]}(s) P_{N} F(s) }{  \cW^{*} (\tau) \widetilde{P}_{N}  G(\tau) }_{ L_{x}^{2} }
	ds  d\tau \r|
	\\
	&\leq  \l| \int_{0}^{\infty}
	\tbra{  \int_{0}^{\infty} e^{-\frac{\tau-s}{2}} \cW^{*} (s)  \1_{[\tau-2^{l+1},\tau-2^{l}]}(s) P_{N} F(s) ds }{  \cW^{*} (\tau) \widetilde{P}_{N}  G(\tau) }_{ L_{x}^{2} }
	d\tau \r|
	\\
	&\leq \int_{0}^{\infty}
	\norm{  \int_{0}^{\infty} e^{-\frac{\tau-s}{2}} \cW^{*} (s)  \1_{[\tau-2^{l+1},\tau-2^{l}]}(s) P_{N} F(s) ds }_{L^{2}}
	\norm{  \widetilde{P}_{N}  G(\tau) }_{ L_{x}^{2} }
	d\tau,
\end{align*}
where we use the Cauchy--Schwarz inequality in the last inequality. Here, by the Strichartz estimates for non-endpoint admissible pair, we get
\begin{align}
\label{eq2.4}
	&\norm{  \int_{0}^{\infty} e^{-\frac{\tau-s}{2}} \cW^{*} (s)  \1_{[\tau-2^{l+1},\tau-2^{l}]}(s) P_{N} F(s) ds }_{L^{2}}
	\\ \notag
	&\cleq N^{\gamma(a)} \norm{ \1_{[\tau-2^{l+1},\tau-2^{l}]} P_{N} F }_{L_{s}^{q(a)'}L_{x}^{a'}},
\end{align}
where $(q(a),a)$ is the non-endpoint admissible pair and $\gamma(a):= d(1/2-1/a)-1/q(a)$ (see \cite[Lemma 2.7]{Inup}). Thus, by the Cauchy--Schwarz inequality, we get
\begin{align*}
	|T_{l}(F,G)|
	&\cleq N^{\gamma(a)}  \int_{0}^{\infty}
	 \norm{ \1_{[\tau-2^{l+1},\tau-2^{l}]} P_{N} F }_{L_{s}^{q(a)'}L_{x}^{a'}}
	\norm{  \widetilde{P}_{N}  G(\tau) }_{ L_{x}^{2} }
	d\tau
	\\
	&\cleq  N^{\gamma(a)}  \l( \int_{0}^{\infty} \norm{ \1_{[\tau-2^{l+1},\tau-2^{l}]} P_{N} F }_{L_{s}^{q(a)'}L_{x}^{a'}}^{2} d\tau \r)^{1/2}
	\norm{  \widetilde{P}_{N}  G }_{ L_{t}^{2} L_{x}^{2} }
\end{align*}
Let $\alpha$ satisfy $1/2+1/\alpha=1/q(a)'$. Then, by the H\"{o}lder inequality, we obtain
\begin{align*}
	&\l( \int_{0}^{\infty} \norm{ \1_{[\tau-2^{l+1},\tau-2^{l}]} P_{N} F }_{L_{t}^{q(a)'}L_{x}^{a'}}^{2} d\tau \r)^{1/2}
	\\
	&\leq \l( \int_{0}^{\infty} \norm{ \1_{[\tau-2^{l+1},\tau-2^{l}]} }_{L_{s}^{\alpha}}^{2}  \norm{ \1_{[\tau-2^{l+1},\tau-2^{l}]} P_{N} F }_{L_{s}^{2}L_{x}^{a'}}^{2} d\tau \r)^{1/2}
	\\
	&\ceq 2^{\frac{1}{\alpha}l} \l( \int_{0}^{\infty} \int_{0}^{\infty}  \1_{[\tau-2^{l+1},\tau-2^{l}]}(s)\norm{ P_{N} F(s) }_{L_{x}^{a'}}^{2} ds d\tau \r)^{1/2}
	\\
	&= 2^{\frac{1}{\alpha}l} \l\{ \int_{0}^{\infty} \l( \int_{[s+2^{l},s+2^{l+1}]}   d\tau \r) \norm{ P_{N} F(s) }_{L_{x}^{a'}}^{2}  ds \r\}^{1/2}
	\\
	&\cleq 2^{\l( \frac{1}{\alpha}+\frac{1}{2} \r)l } \norm{ P_{N} F }_{L_{t}^{2}L_{x}^{a'}}
	\\
	&= 2^{-l\beta(a,2)} \norm{ P_{N} F }_{L_{t}^{2}L_{x}^{a'}}.
\end{align*}
Hence, we get
\begin{align*}
	|T_{l}(F,G)|
	\cleq N^{\gamma(a)+\gamma(2)} 2^{-l\beta(a,2)} \norm{ P_{N} F }_{L_{t}^{2}L_{x}^{a'}} \norm{  \widetilde{P}_{N}  G }_{ L_{t}^{2} L_{x}^{2} }.
\end{align*}

(3). By the symmetry, we get the statement in the same way as (2). 
\end{proof}

By Lemma \ref{lem2.5} and the interpolation lemma (see \cite[p.76, Exercises 5. (b)]{BeLo76} or Lemma 6.1 in \cite{KeTa98}),
we obtain the bilinear estimate in the endpoint case
$(q^*,r^*) = (2,2(d-1)/(d-3))$:
\begin{align*}
	|T(F,G)|
	\cleq N^{2\gamma} \norm{ P_{N}  F }_{L^{(q^*)'}(I:L^{(r^*)'}(\R^d))} \norm{ \widetilde{P}_{N}  G }_{L^{(q^*)'}(I:L^{(r^*)'}(\R^d))}
\end{align*}
for $N \geq 2$.
This part is completely the same as those of Keel--Tao \cite{KeTa98} and Machihara--Nakanishi--Ozawa \cite{MNO03},
and hence, we omit the details.
When $N=1$, the above estimate is true by the estimates in Lemma 2.5 and Remark 1.1. 
Combining these estimates with the Littlewood--Paley decomposition, and by Lemma \ref{lem_bilin_est_nonadm},
we summarize the Strichartz estimates for high frequency part:
\begin{lemma}\label{lem_str_est_high}
Let
$d \ge 2$.
Let
$2 \le r < \infty$ and $2 \le q \le \infty$.
We set
$\gamma := \max\{ d(1/2-1/r)-1/q, \frac{d+1}{2}(1/2-1/r)\}$.
Then, we have
\begin{align*}%
	\left\| \int_0^t \mathcal{D}(t-s) P_{>1} F(s) ds \right\|_{L^{q}(I; L^{r}(\mathbb{R}^d))}
		&\lesssim
		\left\| \langle \nabla \rangle^{2\gamma-1} F \right\|_{L^{q'}(I; L^{r'}(\mathbb{R}^d))}, \\
	\left\| \int_0^t \partial_t \mathcal{D}(t-s) P_{>1} F(s) ds \right\|_{L^{q}(I; L^{r}(\mathbb{R}^d))}
		&\lesssim
		\left\| \langle \nabla \rangle^{2\gamma} F \right\|_{L^{q'}(I; L^{r'}(\mathbb{R}^d))}.
\end{align*}%
\end{lemma}
This lemma and Lemma \ref{lem_str_est_low} immediately imply Theorem \ref{thm1.3}.
Furthermore, applying a duality argument and the Bernstein inequality as in
\cite[Lemmas 2.7--2.10]{Inup}, we obtain Proposition \ref{prop1.4}.

We can also get the Besov type Strichartz estimates. See Appendix \ref{secA} for the proof of Besov type Strichartz estimates, Propositions \ref{prop1.5} and \ref{prop1.6}.


\section{Local well-posedness when $d \geq 6$}
\label{sec3}

We give the local well-posedness when $d \geq 6$ by using the exotic Strichartz estimates. 

\subsection{Function spaces}
\label{sec3.1}

For $d \geq 6$, we define the function spaces as follows.
\begin{align*}
	\norm{u}_{S(I)}
	&:= \norm{u}_{ L_{t,x}^{ \frac{2(d+1)}{d-2} } (I) },
	\\
	\norm{u}_{X(I)}
	&:=\norm{u}_{ L_{t}^{\frac{d^2+d}{d+2}} W_{x}^{ \frac{2}{d}, \frac{2(d+1)}{d-1} }  (I) },
	\\
	\norm{u}_{X'(I)}
	&:=\norm{u}_{ L_{t}^{\frac{d^2+d}{3d+2} } W_{x}^{ \frac{2}{d}, \frac{2(d+1)}{d+3} }  (I) },
	\\
	\norm{u}_{Y(I)}
	&:=\norm{u}_{ L_{t}^{\frac{2d^3-7d^2-9d}{d^3-6d^2+7d-2} } W_{x}^{ \frac{d^2-4d^2}{2d^2-9d}, \frac{4d^3-14d^2-18d}{2d^3-11d^2+11d-8} }  (I) },
	\\
	\norm{u}_{W(I)}
	&:=\norm{u}_{ L_{t}^{\frac{2(d+1)}{d-1}} B_{\frac{2(d+1)}{d-1},2}^{\frac{1}{2}}(I)},
	\\
	 \norm{u}_{W'(I)}
	&:=\norm{u}_{ L_{t}^{\frac{2(d+1)}{d+3}} B_{\frac{2(d+1)}{d+3},2}^{\frac{1}{2}}(I)},
	\\
	\norm{u}_{S^1(I)}
	&:=\max \l\{ \norm{u}_{ L_{t}^{\frac{2d^3-7d^2-9d}{d^3-6d^2+7d-2}}  B_{\frac{4d^3-14d^2-18d}{2d^3-11d^2+11d-8},2}^{ \frac{d^2-4d-2}{2d^2-9d}}}  ,
	\norm{u}_{L_{t}^{\frac{d^2+d}{d+2}} B_{\frac{2d^3-2d}{d^3-5d-8},2}^{ \frac{d^2-2d-2}{d^2-d} }},
	\norm{u}_{W(I)} \r\}.
\end{align*}

We remark that the norms in the definition of $S^1(I)$ have the form $L_{t}^{q} B_{r,2}^{1-\gamma(r)}$ whose $(q,r)$ satisfies $1/q=\frac{d-1}{2}(1/2-1/r)$ and $\gamma(r)=\frac{d+1}{2}(1/2-1/r)$, i.e., $(q,r)$ is a wave admissible pair. 

Since we have $B_{p,2}^{s} \hookrightarrow F_{p,2}^{s} \ceq W^{s,p}$ if $p \geq 2$, we get
\begin{align*}
	\norm{u}_{Y(I)} \cleq \norm{u}_{S^1(I)}.
\end{align*}
Moreover, by the Sobolev embedding and the embedding $B_{p,2}^{s} \hookrightarrow  W^{s,p}$ for $p \geq 2$, we get
\begin{align*}
	\norm{u}_{X(I)} 
	\cleq \norm{u}_{L_{t}^{\frac{d^2+d}{d+2}} W^{ \frac{d^2-2d-2}{d^2-d},\frac{2d^3-2d}{d^3-5d-8} }}
	\cleq \norm{u}_{S^1(I)}.
\end{align*}

We have the following interpolation. 

\begin{lemma}
\label{lem3.1}
Let $d \geq 6$ and $I \subset \R$ be any time interval. Then, we have the following. 

(1). Let $\theta_1= \frac{2d+4}{d^2-2d}$. 
\begin{align*}
	\norm{u}_{X(I)} 
	\cleq \norm{u}_{S(I)}^{\theta_1} \norm{u}_{L_{t}^{\infty} W_{x}^{\frac{2d-4}{d^2-4d-4}, \frac{2d^2-8d-8}{d^2-6d+8}}}^{1-\theta_1} 
	\cleq \norm{u}_{S(I)}^{\theta_1} \norm{u}_{L_{t}^{\infty} H_x^1}^{1-\theta_1}.
\end{align*}

(2). Let $\theta_2= \frac{d}{d^2-3d-4}$. 
\begin{align*}
	\norm{u}_{S(I)} 
	\cleq \norm{u}_{X(I)}^{\theta_2} \norm{u}_{L_{t}^{\frac{2(d+1)}{d-1}} W_{x}^{\frac{1}{2}, \frac{2(d+1)}{d-1}}}^{1-\theta_2} 
	\cleq \norm{u}_{X(I)}^{\theta_2} \norm{u}_{W(I)}^{1-\theta_2}.
\end{align*}

(3). Let $\theta_3= \frac{1}{2(d-4)}$. 
\begin{align*}
	\norm{u}_{L_{t}^{\frac{2(d+1)}{d-2}} W_{x}^{\frac{1}{2}, \frac{2d(d+1)}{d^2-d+1}}} 
	\cleq \norm{u}_{X(I)}^{\theta_3} \norm{u}_{Y(I)}^{1-\theta_3} 
	\cleq \norm{u}_{X(I)}^{\theta_3} \norm{u}_{S^1}^{1-\theta_3}.
\end{align*}
\end{lemma}

\begin{proof}
The statement for the homogeneous spaces are obtained by \cite[Lemma 2.10]{Bu13}. The desired statement can be obtained in the same way as for the homogeneous spaces. 
\end{proof}

\subsection{Nonlinear estimates}
\label{sec3.2}

We collect some nonlinear estimates. 

\begin{lemma}[Estimates for difference]
\label{lem3.2}
Let $1/p=1/p_1+1/p_2=1/p_3+1/p_4$, $1<p_i<\infty$ for $i=1,2,3,4$. Assume that a function $F\in C^{1,\alpha}(\R:\R)$ for some $0<\alpha<1$ and that $F'(0)=0$. Then, we have
\begin{align*}
	\norm{F(u)-F(v)}_{B^{\frac{1}{2}}_{p,2} } 
	\cleq \norm{u-v}_{B^{\frac{1}{2}}_{p_1,2} } \norm{|u|^{\alpha}}_{L^{p_2}}
	+  \norm{|u-v|^\alpha}_{L^{p_3} } \norm{v}_{B^{\frac{1}{2}}_{p_4,2}}.
\end{align*}
\end{lemma}

\begin{proof}
We have
\begin{align*}
	\norm{F(u)-F(v)}_{B^{\frac{1}{2}}_{p,2} } 
	\ceq \norm{F(u)-F(v)}_{\dot{B}^{\frac{1}{2}}_{p,2} } +\norm{F(u)-F(v)}_{L^p }.
\end{align*}
For the first term, we have the following estimate by \cite[Lemma 2.10]{Bu13}.
\begin{align*}
	\norm{F(u)-F(v)}_{\dot{B}^{\frac{1}{2}}_{p,2} } 
	\cleq \norm{u-v}_{\dot{B}^{\frac{1}{2}}_{p_1,2} } \norm{|u|^{\alpha}}_{L^{p_2}}
	+  \norm{|u-v|^\alpha}_{L^{p_3} } \norm{v}_{\dot{B}^{\frac{1}{2}}_{p_4,2}}.
\end{align*}
For the second term, we have the following from the H\"{o}lder inequality and the mean value theorem.
\begin{align*}
	\norm{F(u)-F(v)}_{L^p  } 
	\cleq \norm{u-v}_{L^{p_1}  } \norm{|u|^{\alpha}}_{L^{p_2}}
	+  \norm{|u-v|^\alpha}_{L^{p_3} } \norm{v}_{L^{p_4} }.
\end{align*}
Combining them, we get the desired statement. 
\end{proof}

\begin{lemma}[Nonlinear estimates]
\label{lem3.3}
Let $F(u)=\pm|u|^{\frac{4}{d-2}} u$. Then, the following are true.
\begin{align*}
	\norm{F(u)}_{W'(I)} 
	&\cleq \norm{u}_{X(I)}^{\theta_2\frac{4}{d-2}} \norm{u}_{S^1(I)}^{(1-\theta_2)\frac{4}{d-2}+1},
	\\
	\norm{F(u)}_{X'(I)} 
	&\cleq \norm{u}_{X(I)}^{\theta_2\frac{4}{d-2}+1} \norm{u}_{S^1(I)}^{(1-\theta_2)\frac{4}{d-2}},
	\\
	\norm{\jbra{\nabla}^{\frac{2}{d}} F'(u)}_{L_t^{\frac{d+1}{2}} L_x^{\frac{d^3+d^2}{2d^2+2d+2 }}(I) } 
	&\cleq \norm{u}_{L_t^{\frac{2(d+1)}{d-2}} W^{\frac{1}{2}, \frac{2(d^2+d)}{d^2-d+1}} }^{\frac{4}{d}} \norm{u}_{S(I)}^{\frac{8}{d(d-2)}},
	\\
	\norm{F(u)-F(v)}_{X'(I) } 
	&\cleq \norm{u-v}_{X(I) } \l( \norm{u-v}_{S(I) }^{\frac{4}{d-2}} + \norm{v}_{S(I)}^{\frac{4}{d-2}}   \r) 
	\\
	&\quad + \norm{u-v}_{X(I)} \l( \norm{u-v}_{S(I)} + \norm{v}_{S(I)}\r)^{\frac{8}{d(d-2)}}
	\\
	&\qquad  \times \l( \norm{u-v}_{X(I)}^{\theta_3} \norm{u-v}_{S^1(I)}^{1-\theta_3} + \norm{v}_{X(I)}^{\theta_3} \norm{v}_{Y(I)}^{1-\theta_3}\r)^{\frac{4}{d}},
	\\
	\norm{F(u)-F(v)}_{W'(I)} 
	&\cleq \norm{u-v}_{W(I) } \l( \norm{u-v}_{S(I) }^{\frac{4}{d-2}} + \norm{v}_{S(I)}^{\frac{4}{d-2}}  \r)
	\\
	&\quad + \norm{u-v}_{S(I)}^{\frac{4}{d-2}}  \norm{v}_{W(I)}.
\end{align*}
\end{lemma}

\begin{proof}
By the embedding $W^{s,p} \ceq F_{p,2}^{s}  \hookrightarrow B_{p,2}^{s}$ for $p \leq 2$, the fractional chain rule, the embedding $B_{p,2}^{s} \hookrightarrow  W^{s,p}$ for $p \geq 2$, and Lemma \ref{lem3.1} (2), we know 
\begin{align*}
	\norm{F(u)}_{W'(I)} 
	&\cleq \norm{F(u)}_{ L_{t}^{\frac{2(d+1)}{d+3}} W^{\frac{1}{2},\frac{2(d+1)}{d+3} } }
	\\
	&\cleq \norm{u}_{S(I)}^{\frac{4}{d-2}} \norm{u}_{ L_t^{\frac{2(d+1)}{d-1}} W^{\frac{1}{2}, \frac{2(d+1)}{d-1}} }
	\\
	&\cleq \norm{u}_{S(I)}^{\frac{4}{d-2}} \norm{u}_{ W(I) }
	\\
	&\cleq \norm{u}_{X(I)}^{\theta_2\frac{4}{d-2}} \norm{u}_{S^1(I)}^{(1-\theta_2)\frac{4}{d-2}+1}.
\end{align*}

We have
\begin{align*}
	\norm{F(u)}_{X'(I)}
	& \cleq \norm{u}_{S(I)}^{\frac{4}{d-2}} \norm{u}_{X(I)}
	\\
	& \cleq \norm{u}_{X(I)}^{\theta_2\frac{4}{d-2}+1} \norm{u}_{S^1(I)}^{(1-\theta_2)\frac{4}{d-2}},
\end{align*}
where we used the fractional chain rule and Lemma \ref{lem3.1} (2). 

It is known by \cite{Bu13} that 
\begin{align*}
	\norm{|\nabla|^{\frac{2}{d}} F'(u)}_{L_t^{\frac{d+1}{2}} L_x^{\frac{d^3+d^2}{2d^2+2d+2 }}(I) } 
	&\cleq \norm{u}_{L_t^{\frac{2(d+1)}{d-2}} W^{\frac{1}{2}, \frac{2(d^2+d)}{d^2-d+1}} }^{\frac{4}{d}} \norm{u}_{S(I)}^{\frac{8}{d(d-2)}}.
\end{align*}
It is enough to show that 
\begin{align*}
	\norm{ F'(u)}_{L_t^{\frac{d+1}{2}} L_x^{\frac{d^3+d^2}{2d^2+2d+2 }}(I) } 
	&\cleq \norm{u}_{L_t^{\frac{2(d+1)}{d-2}} L^{\frac{2(d^2+d)}{d^2-d+1}} }^{\frac{4}{d}} \norm{u}_{S(I)}^{\frac{8}{d(d-2)}}.
\end{align*}
This follows from the H\"{o}lder inequality. 

By \cite[Lemma 2.11]{Bu13}, we have
\begin{align*}
	\norm{F(u)-F(v)}_{\dot{X}'(I) } 
	&\cleq \norm{u-v}_{\dot{X}(I) } \l( \norm{u-v}_{S(I) }^{\frac{4}{d-2}} + \norm{v}_{S(I)}^{\frac{4}{d-2}}   \r) 
	\\
	&\quad + \norm{u-v}_{\dot{X}(I)} \l( \norm{u-v}_{S(I)} + \norm{v}_{S(I)}\r)^{\frac{8}{d(d-2)}}
	\\
	&\qquad  \times \l( \norm{u-v}_{\dot{X}(I)}^{\theta_3} \norm{u-v}_{\dot{S}^1(I)}^{1-\theta_3} + \norm{v}_{\dot{X}(I)}^{\theta_3} \norm{v}_{\dot{Y}(I)}^{1-\theta_3}\r)^{\frac{4}{d}}.
\end{align*}
We get these Lebesgue-type inequality by the H\"{o}lder inequality and thus we get the desired inequalities combing them. 

The last inequality follows from Lemma \ref{lem3.2}. 
\end{proof}

\subsection{The proof of L.W.P}
\label{sec3.3}

We prove the local well-posedness when $d \geq 6$.

\begin{proof}
By Lemma \ref{lem3.1} (1), we have
\begin{align*}
	\norm{\cS(t)(u_0,u_1)}_{X(I)} 
	&\cleq \norm{\cS(t)(u_0,u_1)}_{S(I)}^{\theta_1} \norm{\cS(t)(u_0,u_1)}_{L^{\infty}(I:H^1)}^{1-\theta_1}
	\\
	&\cleq \norm{\cS(t)(u_0,u_1)}_{S(I)}^{\theta_1} A^{1-\theta_1},
\end{align*}
where $\cS(t)(u_0,u_1)=\cD(t)(u_0+u_1)+\partial_t \cD(t) u_0$. 
Let $\delta := C\norm{\cS(t)(u_0,u_1)}_{S(I)}^{\theta_1} A^{1-\theta_1}$ and set 
\begin{align*}
	\Phi(u)=\Phi[u_0,u_1](u):= \cS(t)(u_0,u_1) + \int_{0}^{t} \cD(t-s) ( |u(s)|^{\frac{4}{d-2}}u(s) ) ds.
\end{align*}
Take $u \in \{u : \norm{u}_{X(I)} \leq a, \norm{u}_{S^1(I)} \leq b \}$. Then, it follows from the Strichartz estimates and Lemma \ref{lem3.3} that 
\begin{align*}
	\norm{\Phi(u)}_{X(I)} 
	&\leq \delta +  \norm{\int_{0}^{t} \cD(t-s) ( |u(s)|^{\frac{4}{d-2}}u(s) ) ds}_{X(I)} 
	\\
	&\leq \delta + C \norm{ |u|^{\frac{4}{d-2}}u}_{X'(I)}
	\\
	&\leq \delta +C \norm{u}_{X(I)}^{\theta_2\frac{4}{d-2}+1} \norm{u}_{S^1(I)}^{(1-\theta_2)\frac{4}{d-2}}
	\\
	&\leq \delta+Ca^{\theta_2\frac{4}{d-2}+1} b^{(1-\theta_2)\frac{4}{d-2}}
\end{align*}
and 
\begin{align*}
	\norm{\Phi(u)}_{S^1(I)} 
	&\leq \norm{\cS(t)(u_0,u_1)}_{S^1(I)} +  \norm{\int_{0}^{t} \cD(t-s) ( |u(s)|^{\frac{4}{d-2}}u(s) ) ds}_{S^1(I)} 
	\\
	&\leq CA + C \norm{ |u|^{\frac{4}{d-2}}u}_{W'(I)}
	\\
	&\leq CA + C\norm{u}_{X(I)}^{\theta_2\frac{4}{d-2}} \norm{u}_{S^1(I)}^{(1-\theta_2)\frac{4}{d-2}+1}
	\\
	&\leq CA + C a^{\theta_2\frac{4}{d-2}}b^{(1-\theta_2)\frac{4}{d-2}+1}.
\end{align*}
Therefore, taking $a=2\delta$ and $b=2CA$ and choosing sufficiently small $\delta$ such that $C a^{\theta_2\frac{4}{d-2}}b^{(1-\theta_2)\frac{4}{d-2}} \leq 1/2$, it follows that
\begin{align*}
	\norm{\Phi(u)}_{X(I)} 
	\leq \delta+Ca^{\theta_2\frac{4}{d-2}+1} b^{(1-\theta_2)\frac{4}{d-2}}
	\leq a,
\end{align*}
and 
\begin{align*}
	\norm{\Phi(u)}_{S^1(I)} 
	\leq CA + C a^{\theta_2\frac{4}{d-2}}b^{(1-\theta_2)\frac{4}{d-2}+1}
	\leq b.
\end{align*}
Thus, $\Phi$ is a mapping on $\{u : \norm{u}_{X(I)} \leq a, \norm{u}_{S^1(I)} \leq b \}$. For $u,v \in \{u : \norm{u}_{X(I)} \leq a, \norm{u}_{S^1(I)} \leq b \}$, by the Strichartz estimate and Lemma \ref{lem3.3}, we have
\begin{align*}
	\norm{\Phi(u)-\Phi(v)}_{X(I)} 
	&\cleq 
	\norm{F(u)-F(v)}_{X'(I) } 
	\\
	&\cleq \norm{u-v}_{X(I) } \l( \norm{u-v}_{S(I)}^{\frac{4}{d-2}} + \norm{v}_{S(I)}^{\frac{4}{d-2}}   \r) 
	\\
	&\quad + \norm{u-v}_{X(I)} \l( \norm{u-v}_{S(I)} + \norm{v}_{S(I)}\r)^{\frac{8}{d(d-2)}}
	\\
	&\qquad  \times \l( \norm{u-v}_{X(I)}^{\theta_3} \norm{u-v}_{S^1(I)}^{1-\theta_3} + \norm{v}_{X(I)}^{\theta_3} \norm{v}_{Y(I)}^{1-\theta_3}\r)^{\frac{4}{d}}
	\\
	&=:I+I\!I.
\end{align*}
It follows from Lemma \ref{lem3.1} (2) that
\begin{align*}
	I 
	&\cleq  \norm{u-v}_{X(I)} \l( \norm{u-v}_{X(I)}^{\theta_2\frac{4}{d-2}} \norm{u-v}_{S^1(I)}^{(1-\theta_2)\frac{4}{d-2}}
	+ \norm{v}_{X(I)}^{\theta_2\frac{4}{d-2}} \norm{v}_{S^1(I)}^{(1-\theta_2)\frac{4}{d-2}}  \r) 
	\\
	&\cleq  \norm{u-v}_{X(I)} a^{\theta_2\frac{4}{d-2}} b^{(1-\theta_2)\frac{4}{d-2}}.
\end{align*}
We also have
\begin{align*}
	I\!I \cleq \norm{u-v}_{X(I)} (a^{\theta_2} b^{1-\theta_2})^{\frac{8}{d(d-2)}} (a^{\theta_3} b^{1-\theta_3})^{\frac{4}{d}}.
\end{align*}
Therefore, if $\delta$ (\textit{i.e.} $a$) is small, then $\Phi$ is a contraction map on $\{u : \norm{u}_{X(I)} \leq a, \norm{u}_{S^1(I)} \leq b \}$ with the distance $d(u,v):=\norm{u-v}_{X(I)}$. From the contraction mapping principle, we get the unique solution. 
\end{proof}


\section{Unconditional uniqueness}
\label{sec4}

\subsection{Paraproduct estimates}
\label{sec4.1}

We show paraproduct estimates for the inhomogeneous Besov spaces. 

We have an equivalence of Besov norms. See Appendix \ref{secB} for the proof. 

\begin{lemma}\label{lem_41}
Let  $1<p < \infty$, $1\leq q \leq \infty$, and $s >0$. Then it holds that
\begin{align*}
	\| f \|_{B_{p,q}^{s}} 
	&\sim_{J} \| \Delta_{\leq J} f \|_{L^p}  +  \left\|  \{ 2^{js} \| \Delta_{\geq j} f \|_{L^p}  \}_{j=J}^{\infty} \right\|_{l^{q}},
	\\
	\| f \|_{B_{p,q}^{-s}} 
	&\sim_{J} \| \Delta_{\leq J} f \|_{L^p}  +  \left\|  \{ 2^{-js} \| \Delta_{\leq j} f \|_{L^p}  \}_{j=J}^{\infty} \right\|_{l^{q}},
\end{align*}
where the implicit constants may depend on $J \in \mathbb{Z}_{\geq 0}$. 
\end{lemma}

We decompose the product of two functions $f$ and $g$ in the following way.
\begin{align*}
	fg
	&=\Delta_{\leq 0} (fg) + \sum_{j=1}^{\infty} \Delta_{j} (fg)
	\\
	&=\Delta_{\leq 0} ((\Delta_{\leq 3}f)g) + \Delta_{\leq 0}((\Delta_{>3}f)g)
	\\
	&\quad + \sum_{j=1}^{\infty} \Delta_{j}  ((\Delta_{\leq j+ 3}f)g) + \sum_{j=1}^{\infty} \Delta_{j} ((\Delta_{>j+3}f)g)
	\\
	&=\Delta_{\leq 0} ((\Delta_{\leq 3}f)g) + \Delta_{\leq 0}((\Delta_{>3}f)\Delta_{>1}g)
	\\
	&\quad + \sum_{j=1}^{\infty} \Delta_{j}  ((\Delta_{\leq j+ 3}f)g) + \sum_{j=1}^{\infty} \Delta_{j} ((\Delta_{>j+3}f)\Delta_{> j+1}g).
\end{align*}
We set
\begin{align}
\label{g1}
	G_{1}(f,g)&:= \Delta_{\leq 0} ((\Delta_{\leq 3}f)g) + \sum_{j=1}^{\infty} \Delta_{j}  ((\Delta_{\leq j+ 3}f)g),
	\\
\label{g2}
	G_{2}(f,g)&:=\Delta_{\leq 0}((\Delta_{>3}f)\Delta_{>1}g) + \sum_{j=1}^{\infty} \Delta_{j} ((\Delta_{>j+3}f)\Delta_{> j+1}g).
\end{align}
Namely, we have the identity
$fg = G_{1}(f,g) + G_{2}(f,g)$.
We prepare the following paraproduct estimates.

\begin{lemma}[paraproduct estimates]\label{lem_pp}
Let $1<p_i<\infty$, $i=1,2,\cdots, 8$, $s>0$, and $\sigma>0$. Then, we have
\begin{align}
\label{g1_est}
	&\| G_{1}(f,g) \|_{B_{p,2}^{-s}} \lesssim \| f \|_{B_{p_1,2}^{-s}} \| g \|_{L^{p_2}} 
	\qquad \text{ if } \frac{1}{p}=\frac{1}{p_1} + \frac{1}{p_2},
	\\
\label{g2_est1}
	&\| G_{2}(f,g) \|_{B_{p,2}^{-s}} \lesssim \| f \|_{B_{p_3,2}^{-s}} \| g \|_{B_{p_4,\infty}^{s_1}} 
	\quad \text{ if } s_1>s \text{ and } \frac{1}{p}+\frac{s_1}{d}=\frac{1}{p_3} + \frac{1}{p_4},
	\\
\label{g2_est2}
	&\| G_{2}(f,g) \|_{B_{p,2}^{\sigma}} \lesssim \| f \|_{B_{p_5,2}^{-s}} \| g \|_{B_{p_6,\infty}^{s+\sigma}} 
	\quad \text{ if } \frac{1}{p}=\frac{1}{p_5} + \frac{1}{p_6}.
\end{align}
Moreover, when $g=f$, we have
\begin{align}
\label{g1_est2}
	&\| G_{1}(f,f) \|_{B_{p,2}^{\sigma}} \lesssim \| f \|_{B_{p_7,2}^{-s}} \| g \|_{B_{p_8,\infty}^{s+\sigma}} 
	\quad \text{ if } \frac{1}{p}=\frac{1}{p_7} + \frac{1}{p_8}.
\end{align}
\end{lemma}

\begin{proof}
We first prove \eqref{g1_est}.
By the equivalence of the Besov norm, we get
\begin{align*}
	\| G_{1}(f,g) \|_{B_{p,2}^{-s}}
	&= \| \Delta_{\leq 0} G_{1}(f,g) \|_{L^p}  +  \left\|  \{ 2^{-js} \| \Delta_{j} G_{1}(f,g) \|_{L^p}  \}_{j=1}^{\infty} \right\|_{l^{2}}
	\\
	&\lesssim \left\| \Delta_{\leq 0}^2 ((\Delta_{\leq 3}f)g)  \right\|_{L^p}  
	+ \left\| \Delta_{\leq 0}  \Delta_{1}  ((\Delta_{\leq 4}f)g) \} \right\|_{L^p}
	\\
	&\quad + 2^{-s} \| \Delta_{1} \Delta_{\leq 0} ((\Delta_{\leq 3}f)g) \|_{L^p}
	\\
	&\quad + \left\|  \{ 2^{-js} \| \Delta_{j}  \sum_{k=j-1}^{j+1} \Delta_{k}  ((\Delta_{\leq k+ 3}f)g) ] \|_{L^p}  \}_{j=1}^{\infty}  \right\|_{l^{2}}
	 \\
	 &=:A+B+C+D.
\end{align*}
It follows from the $L^p$-boundedness of the projections $\Delta_j$ and $\Delta_{\leq j}$ and the H\"{o}lder inequality that 
\begin{align*}
	A + B + C \lesssim \| (\Delta_{\leq 4}f)\|_{L^{p_1}} \|g \|_{L^{p_2}}.
\end{align*}
Moreover, we have
\begin{align*}
	D&=  \left\|  \{ 2^{-js} \| \Delta_{j}  \sum_{k=j-1}^{j+1} \Delta_{k}  ((\Delta_{\leq k+ 3}f)g) ] \|_{L^p}  \}_{j=1}^{\infty}  \right\|_{l^{2}}
	\\
	&\lesssim  \left\|  \{  \sum_{k=j-1}^{j+1}  2^{-js} \| (\Delta_{\leq k+ 3}f) \|_{L^{p_1}}  \}_{j=1}^{\infty} \right\|_{l^{2}}  \|g \|_{L^{p_2}}
	\\
	 &\lesssim \left\|  \{   2^{-js} \| (\Delta_{\leq j+4}f) \|_{L^{p_1}}  \}_{j=1}^{\infty} \right\|_{l^{2}}  \|g \|_{L^{p_2}}
	  \\
	 &\lesssim \left\|  \{   2^{-js} \| (\Delta_{\leq j}f) \|_{L^{p_1}}  \}_{j=4}^{\infty} \right\|_{l^{2}}  \|g \|_{L^{p_2}}.
\end{align*}
Combining the estimates of $A,B,C$, and $D$, we obtain
\begin{align*}
	\| G_{1}(f,g) \|_{B_{p,2}^{-s}} 
	&\lesssim  \left( \| (\Delta_{\leq 4}f)\|_{L^{p_1}}  + \left\|  \{ 2^{-js} \| (\Delta_{\leq j}f) \|_{L^{p_1}}  \}_{j=4}^{\infty} \right\|_{l^{2}} \right)  \|g \|_{L^{p_2}}
	\\
	&\sim  \| f \|_{B_{p_1,2}^{-s}} \| g \|_{L^{p_2}}.
\end{align*}
Next, we show \eqref{g2_est1}.
In the same way as before, we have
\begin{align*}
	\| G_{2}(f,g) \|_{B_{p,2}^{-s}} 
	&\lesssim \|  \Delta_{\leq 0}((\Delta_{>3}f)\Delta_{>1}g)\|_{L^p}  
	+  \| \Delta_{1} \Delta_{\leq 0}((\Delta_{>3}f)\Delta_{>1}g) \|_{L^p} 
	\\
	&+\| \Delta_{\leq 0}\Delta_{1} ((\Delta_{>4}f)\Delta_{> 2}g)\|_{L^p}  
	\\
	&+  \left\|  \{ 2^{-js} \| \Delta_{j} \sum_{k=j-1}^{j+1} \Delta_{k} ((\Delta_{>k+3}f)\Delta_{> k+1}g) \|_{L^p}  \}_{j=1}^{\infty} \right\|_{l^{2}}
	\\
	&=:A+B+C+D.
\end{align*}
Noting the supports of
$\widehat{\Delta_{>k+3}f}$ and $\widehat{\Delta_{> k+1}g}$,
we have
\begin{align*}
	D
	&=  \left\|  \{ 2^{-js} \| \sum_{k=j-1}^{j+1} \sum_{l>k+3,|l-l'|\leq 2}  \Delta_{j} \Delta_{k} ((\Delta_{l}f)\Delta_{l'}g) \|_{L^p}  \}_{j=1}^{\infty} \right\|_{l^{2}}
	\\
	&\lesssim \left\|  \{ \sum_{l>j+2,|l-l'|\leq 2}  2^{-js} 2^{js_1} \|  ((\Delta_{l}f)\Delta_{l'}g) \|_{L^{\frac{pd}{d+ps_1}}}  \}_{j=1}^{\infty} \right\|_{l^{2}}
	\\
	&\lesssim \left\|  \{ \sum_{l>j+2,|l-l'|\leq 2}  2^{-js} 2^{js_1} \| \Delta_{l}f \|_{L^{p_3}} \| \Delta_{l'}g \|_{L^{p_4}}  \}_{j=1}^{\infty} \right\|_{l^{2}}
	\\
	&\lesssim \left\|  \{ \sum_{l>j+2,|l-l'|\leq 2}  2^{(j-l)(s_1-s)} 2^{-ls} \| \Delta_{l}f \|_{L^{p_3}} 2^{l's_1} \| \Delta_{l'}g \|_{L^{p_4}}  \}_{j=1}^{\infty} \right\|_{l^{2}}
	\\
	&\lesssim \left\|  \{ \sum_{l>j+2}  2^{(j-l)(s_1-s)} 2^{-ls} \| \Delta_{l}f \|_{L^{p_3}} \}_{j=1}^{\infty}  \right\|_{l^{2}}
	\|  2^{l's_1} \| \Delta_{l'}g \|_{L^{p_4}} \|_{l^{\infty}(\mathbb{N})}
	\\
	&\lesssim \left\|  \{ \sum_{l>j+2}  2^{(j-l)(s_1-s)} 2^{-ls} \| \Delta_{l}f \|_{L^{p_3}} \}_{j=1}^{\infty}  \right\|_{l^{2}}
	\|  g \|_{B_{p_4,\infty}^{s_1}}.
\end{align*}
By noting $s_1>s$ and the Young inequality, we get 
\begin{align*}
	&\left\|  \{ \sum_{l>j+2}  2^{(j-l)(s_1-s)} 2^{-ls} \| \Delta_{l}f \|_{L^{p_3}} \}_{j=1}^{\infty}  \right\|_{l^{2}}
	\lesssim \| f \|_{B_{p_3,2}^{-s}}.
\end{align*}
In the same way as above, we can calculate $A,B,C$ and have
\begin{align*}
	A + B + C \lesssim \| f \|_{B_{p_3,2}^{-s}} \| g \|_{B_{p_4,\infty}^{s_1}}.
\end{align*}
Therefore, we obtain \eqref{g2_est1}.
The inequality \eqref{g2_est2} can be obtained in a similar way. We omit the detail. 

Finally, we remark on the proof of \eqref{g1_est2}.
As before, we write
\begin{align*}
	\| G_1(f,f) \|_{B^{\sigma}_{p,2}}
	&= \| \Delta_{\le 0}^2 ((\Delta_{\le 3}f)f) \|_{L^p}
		+ \| \Delta_{\le 0} \Delta_1 ((\Delta_{\le 4}f)f) \|_{L^p} \\
	&\quad + 2^{\sigma} \| \Delta_1 \Delta_{\le 0} ((\Delta_{\le 3}f)f) \|_{L^p} \\
	&\quad + \left\| \left\{
		2^{j \sigma} \left\| \Delta_j \sum_{k=j-1}^{j+1} \Delta_k ((\Delta_{\le k+3}f)f) \right\|_{L^p}
			\right\}_{j=1}^{\infty} \right\|_{l^2} \\
	&=: A+B+C+D.
\end{align*}
We immediately obtain
\begin{align*}
	A + B + C \lesssim \| \Delta_{\le 3} f \|_{L^{p_7}} \| \Delta_{\le 3} f \|_{L^{p_8}}.
\end{align*}
For the term $D$,
noting the support of
$\mathcal{F}(\Delta_j \Delta_k ((\Delta_{\le k+3}f)f))$,
we calculate
\begin{align*}
	D&\lesssim
		\left\| \left\{
		2^{j \sigma} \left\| \Delta_j (( \Delta_{\le j+3} f) \Delta_j f) \right\|_{L^p}
		\right\}_{j=1}^{\infty} \right\|_{l^2} \\
	&\lesssim
		\left\| \left\{
		2^{-sj} \| \Delta_{\le j+3} f \|_{L^{p_7}}
		2^{(s+\sigma)j} \| \Delta_j f \|_{L^{p_8}} \right\}_{j=1}^{\infty} \right\|_{l^2} \\
	&\lesssim
		\left\| \left\{
		2^{-sj} \| \Delta_{\le j+3} f \|_{L^{p_7}} \right\}_{j=1}^{\infty} \right\|_{l^2}
		\left\| \left\{ 2^{(s+\sigma)j} \| \Delta_j f \|_{L^{p_8}} \right\}_{j=1}^{\infty} \right\|_{l^{\infty}} \\
	&\lesssim
		\| f \|_{B^{-s}_{p_7, 2}} \| f \|_{B^{s+\sigma}_{p_8,\infty}}.
\end{align*}
Here, we have used Lemma \ref{lem_41} for the last inequality.
\end{proof}

\subsection{Proof of the unconditional uniqueness}
We give the proof of Theorem \ref{thm_uu}.
The proof is
based on the Besov-type inhomogeneous endpoint Strichartz estimates and paraproduct estimates.
The argument is almost the same as those of
\cite[Theorem 3.4]{Bu13} and \cite[Proposition 2]{Pl03}.
However, for readers' convenience, we give a complete proof.
\begin{proof}
We assume that $u$ and $v$ are solutions to \eqref{NLDW}
on a time interval $I$
with the same initial data
$(u_0, u_1)\in H^1(\mathbb{R}^d)\times L^2(\mathbb{R}^d)$
in the sense of Definition \ref{def1.1}.
By the proof of Theorem \ref{thm_lwp},
we have construct a local solution $u$ satisfying
\begin{align}%
\label{4_sol_est}
	\| u \|_{L_t^{\frac{2(d+1)}{d-2}}(I_0 ;
						B^{\frac{d}{2(d-1)}}_{\frac{2(d^2-1)}{d^2-2d+3},2}(\mathbb{R}^d))}
		\lesssim \| (u_0, u_1) \|_{H^1\times L^2},
\end{align}%
because
$(\frac{2(d+1)}{d-2}, \frac{2(d^2-1)}{d^2-2d+3})$
is admissible with $1-\gamma = \frac{d}{2(d-1)}$ and $\delta =0$. 
Therefore, we may assume this bound for $u$ without loss of generality.

In the following, we divide the proof into the cases $d \ge 5$ and $d =4$.
First, we treat the case $d\ge 5$.
Putting $F(u) = |u|^{\frac{4}{d-2}}u$ and $w = u-v$, we have
\begin{align}%
\label{4_w_eq}
	w(t,x) = \int_0^t \mathcal{D}(t-s) \left( F(u) - F(v) \right) \,ds.
\end{align}%
To express the nonlinear term by $u$ and $w$, we calculate
\begin{align*}%
	F(u) - F(v) &=
		w \int_0^1 F'(\lambda u + (1-\lambda) v)\,d\lambda\\
		&= w \int_0^1 \left( F'(u-(1-\lambda)w) - F'(u) \right)\,d\lambda + w F'(u)\\
		&= w H(u,w) + w F'(u),
\end{align*}%
where
$H := \int_0^1 \left( F'(u-(1-\lambda)w) - F'(u) \right)\,d\lambda$.
Since
$F'(u) = (1+\frac{4}{d-2})|u|^{\frac{4}{d-2}}$
is
$\frac{4}{d-2}$-H\"{o}lder continuous on $\mathbb{R}$, we obtain
\begin{align}%
\label{4_estH}
	|H(u,w)| &\lesssim
		\int_0^1 | (1-\lambda) w |^{\frac{4}{d-2}} \,d\lambda
		\lesssim |w|^{\frac{4}{d-2}}.
\end{align}%
Let $I_0$ be a small time interval such that $0 \in I_0$ determined later.
We decompose the product
$w H(u,w)$
into
$w H (u,w) = G_1(w, H(u,w)) + G_2(w, H(u,w))$,
where
$G_1$ and $G_2$ are paraproducts defined by \eqref{g1} and \eqref{g2}, respectively.
in the same way, we also decompose
$w F'(u)$
into
$w F'(u) = G_1(w, F'(u)) + G_2(w,F'(u))$.
By applying Proposition \ref{prop1.6} to \eqref{4_w_eq} with
$s = -\frac{1}{d-1}$, $r = \frac{2(d-1)}{d-3}, q=2$, we have
\begin{align}%
\label{4_w_str}
	\| w \|_{L^2_t (I_0 ; B^{-\frac{1}{d-1}}_{\frac{2(d-1)}{d-3}, 2} (\mathbb{R}^d))}
	&\lesssim
		\| G_1(w, F'(u)) \|_{L_t^{\frac{2(d+1)}{d+5}} (I_0 ;
						B^{-\frac{1}{d-1}}_{\frac{2(d^2-1)}{d^2+2d-7},2}(\mathbb{R}^d))}\\
\notag
	&\quad +
	\| G_2(w, F'(u)) \|_{L_t^{\frac{2(d+1)}{d+5}} (I_0 ;
						B^{-\frac{1}{d-1}}_{\frac{2(d^2-1)}{d^2+2d-7},2}(\mathbb{R}^d))}\\
\notag
	&\quad +
	\| G_1(w, H(u,w)) \|_{L_t^{2} (I_0 ;
						B^{\frac{1}{d-1}}_{\frac{2(d-1)}{d+1},2}(\mathbb{R}^d))}\\
\notag
	&\quad +
	\| G_2(w, H(u,w)) \|_{L_t^{2} (I_0 ;
						B^{\frac{1}{d-1}}_{\frac{2(d-1)}{d+1},2}(\mathbb{R}^d))}.
\end{align}%
Here, we note that
$\gamma = \frac{d+1}{2(d-1)}$,
and in the first and second terms in RHS, we have taken
$\tilde{q}' = \frac{2(d+1)}{d+5}$, $\tilde{r}' = \frac{2(d^2-1)}{d^2+2d-7}$
(in that case, $\delta = 0$ and $\tilde{\gamma} = \frac{d-3}{2(d-1)}$, which give
$\gamma + \tilde{\gamma} +\delta - 1 = 0$).
Also, in the third and forth terms in RHS, we have taken
$\tilde{q}' = 2$, $\tilde{r}' = \frac{2(d-1)}{d+1}$
(in that case, $\delta = 0$ and $\tilde{\gamma} = \frac{d+1}{2(d-1)}$, which give
$\gamma + \tilde{\gamma} + \delta - 1 = \frac{2}{d-1}$).

We first give the estimate for 
$G_1(w, F'(u))$
in \eqref{4_w_str}.
Applying \eqref{g1_est} of Lemma \ref{lem_pp} with
$s = \frac{1}{d-1}$,
$p = \frac{2(d^2-1)}{d^2+2d-7}$,
$p_1 = \frac{2(d-1)}{d-3}$,
and
$p_2 = \frac{d+1}{2}$,
and then using the H\"{o}lder inequality in time with
$\frac12 + \frac{2}{d+1} = \frac{d+5}{2(d+1)}$,
we see that
\begin{align*}%
	\| G_1(w, F'(u)) \|_{L_t^{\frac{2(d+1)}{d+5}} (I_0 ;
						B^{-\frac{1}{d-1}}_{\frac{2(d^2-1)}{d^2+2d-7},2}(\mathbb{R}^d))}
	&\lesssim
		\| w \|_{L_t^2(I_0 ; B^{-\frac{1}{d-1}}_{\frac{2(d-1)}{d-3},2}(\mathbb{R}^d))}
		\| F'(u) \|_{L_t^{\frac{d+1}{2}}(I_0 ; L_x^{\frac{d+1}{2}}(\mathbb{R}^d))} \\
	&\lesssim
		\| w \|_{L_t^2(I_0 ; B^{-\frac{1}{d-1}}_{\frac{2(d-1)}{d-3},2}(\mathbb{R}^d))}
		\| u \|_{L_t^{\frac{2(d+1)}{d-2}}( I_0 ; L_x^{\frac{2(d+1)}{d-2}}(\mathbb{R}^d))}^{\frac{4}{d-2}} \\
	&\lesssim
		\| w \|_{L_t^2(I_0 ; B^{-\frac{1}{d-1}}_{\frac{2(d-1)}{d-3},2}(\mathbb{R}^d))}
		\| u \|_{L_t^{\frac{2(d+1)}{d-2}}(I_0 ;
							B^{\frac{d}{2(d-1)}}_{\frac{2(d^2-1)}{d^2-2d+3}, 2}(\mathbb{R}^d))}^{\frac{4}{d-2}}.
\end{align*}%
Here, in the second and third lines we have used
$|F'(u)| \lesssim |u|^{\frac{4}{d-2}}$ and the embedding
$B^{\frac{d}{2(d-1)}}_{\frac{2(d^2-1)}{d^2-2d+3}, 2}(\mathbb{R}^d) \subset L^{\frac{2(d+1)}{d-2}}(\mathbb{R}^d)$,
respectively.

Secondly, we give the estimate for $G_2(w, F'(u))$
in \eqref{4_w_str}.
Applying \eqref{g2_est1} of Lemma \ref{lem_pp} with
$s = \frac{2}{d-1}$,
$p = \frac{2(d^2-1)}{d^2+2d-7}$,
$p_3 = \frac{2(d-1)}{d-3}$,
$p_4 = \frac{d(d^2-1)}{2(d^2+1)}$,
and
$s_1 = \frac{2}{d-1}$,
and then using the H\"{o}lder inequality in time with
$\frac12 + \frac{2}{d+1} = \frac{d+5}{2(d+1)}$,
we see that
\begin{align*}%
	\| G_2(w, F'(u)) \|_{L_t^{\frac{2(d+1)}{d+5}} (I_0 ; 
					B^{-\frac{1}{d-1}}_{\frac{2(d^2-1)}{d^2+2d-y},2} (\mathbb{R}^d))}
	&\lesssim
	\| w \|_{L_t^2 (I_0 ; B^{-\frac{1}{d-1}}_{\frac{2(d-1)}{d-3},2}(\mathbb{R}^d))}
	\| F'(u) \|_{L_t^{\frac{d+1}{2}}(I_0 ; B^{\frac{2}{d-1}}_{\frac{d(d^2-1)}{2(d^2+1)},\infty}(\mathbb{R}^d))} \\
	&\lesssim
	\| w \|_{L_t^2 (I_0 ; B^{-\frac{1}{d-1}}_{\frac{2(d-1)}{d-3},2}(\mathbb{R}^d))}
	\| u \|_{L_t^{\frac{2(d+1)}{d-2}}(I_0 ;
			B^{\frac{d-2}{2(d-1)}}_{\frac{2d(d^2-1)}{(d-2)(d^2+1)},\infty} (\mathbb{R}^d))}^{\frac{4}{d-2}} \\
	&\lesssim
	\| w \|_{L_t^2 (I_0 ; B^{-\frac{1}{d-1}}_{\frac{2(d-1)}{d-3},2}(\mathbb{R}^d))}
	\| u \|_{L_t^{\frac{2(d+1)}{d-2}}(I_0 ;
			B^{\frac{d}{2(d-1)}}_{\frac{2(d^2-1)}{d^2-2d+3},2} (\mathbb{R}^d))}^{\frac{4}{d-2}}.
\end{align*}%
Here, in the second and third lines we have used
$|F'(u)| \lesssim |u|^{\frac{4}{d-2}}$ and the embedding
$B^{\frac{d}{2(d-1)}}_{\frac{2(d^2-1)}{d^2-2d+3},2} (\mathbb{R}^d)
\subset B^{\frac{d-2}{2(d-1)}}_{\frac{2d(d^2-1)}{(d-2)(d^2+1)},\infty} (\mathbb{R}^d)$,
respectively.

Thirdly, we give the estimate for $G_2(w, H(u,w))$
in \eqref{4_w_str}.
Applying \eqref{g2_est2} of Lemma \ref{lem_pp} with
$\sigma = \frac{1}{d-1}$,
$p = \frac{2(d-1)}{d+1}$,
$p_5 = \frac{2(d-1)}{d-3}$,
$p_6 = \frac{d-1}{2}$,
and
$s = \frac{1}{d-1}$,
and then using the H\"{o}lder inequality in time with
$\frac{1}{2} = \frac{1}{2} + \frac{1}{\infty}$, we see that
\begin{align*}%
	\| G_2(w, H(u,w)) \|_{L_t^2 (I_0 ; B^{\frac{1}{d-1}}_{\frac{2(d-1)}{d+1},2} (\mathbb{R}^d))}
	&\lesssim
	\| w \|_{L_t^2 (I_0 ; B^{-\frac{1}{d-1}}_{\frac{2(d-1)}{d-3},2}(\mathbb{R}^d))}
	\| H(u,w) \|_{L_t^{\infty}(I_0 ; B^{\frac{2}{d-1}}_{\frac{d-1}{2},\infty}(\mathbb{R}^d))}.
\end{align*}%
By the Gagliardo--Nirenberg interpolation inequality,
the last term of RHS is further estimated as
\begin{align*}%
	&\| H(u,w) \|_{L_t^{\infty}(I_0 ; B^{\frac{2}{d-1}}_{\frac{d-1}{2},\infty}(\mathbb{R}^d))} \\
	&\lesssim
	\| H(u,w) \|_{L_t^{\infty}(I_0 ; B^{0}_{\frac{d}{2}, \infty}(\mathbb{R}^d))}^{\frac12}
	\| H(u,w) \|_{L_t^{\infty}(I_0 ; B^{\frac{4}{d-1}}_{\frac{d(d-1)}{2(d+1)},\infty} (\mathbb{R}^d))}^{\frac12}.
\end{align*}%
By \eqref{4_estH} and Lemma \ref{lem_b} in Appendix C with $f(z) = |z|^{\frac{4}{d-2}}$,
we estimate
\begin{align*}%
	\| H(u,w) \|_{L_t^{\infty}(I_0 ; B^{0}_{\frac{d}{2}, \infty}(\mathbb{R}^d))}^{\frac12}
	&\lesssim
	\| w \|_{L_t^{\infty}(I_0 ; B^{0}_{\frac{2d}{d-2}, \infty}(\mathbb{R}^d))}^{\frac{2}{d-2}}\\
	&\lesssim
	\| w \|_{L_t^{\infty}(I_0 ; H^1(\mathbb{R}^d)}^{\frac{2}{d-2}},
\end{align*}%
and
\begin{align*}%
	&\| H(u,w) \|_{L_t^{\infty}(I_0 ; B^{\frac{4}{d-1}}_{\frac{d(d-1)}{2(d+1)},\infty} (\mathbb{R}^d))}^{\frac12}\\
	&\lesssim
	\left( \int_0^1
		\| F'(u-(1-\lambda)w) - F'(u) \|_{L_t^{\infty}(I_0 ; B^{\frac{4}{d-1}}_{\frac{d(d-1)}{2(d+1)},\infty}(\mathbb{R}^d))}
		\,d\lambda \right)^{\frac12} \\
	&\lesssim
	\left( \int_0^1 \left( \| u-(1-\lambda)w \|_{L_t^{\infty}(I_0 ;
							B^{\frac{d-2}{d-1}}_{\frac{2d(d-1)}{(d-2)(d+1)},\infty}(\mathbb{R}^d))}^{\frac{4}{d-2}}
				+ \| u \|_{L_t^{\infty}(I_0 ;
							B^{\frac{d-2}{d-1}}_{\frac{2d(d-1)}{(d-2)(d+1)},\infty}(\mathbb{R}^d))}^{\frac{4}{d-2}}
				\right)\,d\lambda \right)^{\frac12} \\
	&\lesssim
	\| w \|_{L_t^{\infty}(I_0 ; H^1(\mathbb{R}^d))}^{\frac{2}{d-2}}
	+ \| u \|_{L_t^{\infty}(I_0 ; H^1(\mathbb{R}^d))}^{\frac{2}{d-2}}.
\end{align*}%
Here, we have also used the embedding
$H^1(\mathbb{R}^d) \subset B^{0}_{\frac{2d}{d-2}, \infty}(\mathbb{R}^d)$
and
$H^1(\mathbb{R}^d) \subset B^{\frac{d-2}{d-1}}_{\frac{2d(d-1)}{(d-2)(d+1)},\infty}(\mathbb{R}^d)$.
Consequently, we have
\begin{align}%
\label{4_g2H_est}
	\| G_2(w, H(u,w)) \|_{L_t^2 (I_0 ; B^{\frac{1}{d-1}}_{\frac{2(d-1)}{d+1},2} (\mathbb{R}^d))}
	&\lesssim
	\| w \|_{L_t^2 (I_0 ; B^{-\frac{1}{d-1}}_{\frac{2(d-1)}{d-3},2}(\mathbb{R}^d))}
	\| w \|_{L_t^{\infty}(I_0 ; H^1(\mathbb{R}^d))}^{\frac{2}{d-2}} \\
\notag
	&\quad \times
	\left( \| w \|_{L_t^{\infty}(I_0 ; H^1(\mathbb{R}^d))}^{\frac{2}{d-2}}
	+ \| u \|_{L_t^{\infty}(I_0 ; H^1(\mathbb{R}^d))}^{\frac{2}{d-2}} \right).
\end{align}%

Finally, we give the estimate for $G_1(w, H(u,w))$
in \eqref{4_w_str}.
We further decompose $G_1(w, H(u,w))$ into
\begin{align*}%
	G_1(w, H(u,w))
	&= \sum_{j \in \mathbb{Z}} \Delta_j \left( (\Delta_{\le j+3} w ) H(u,w) \right)\\
	&= \sum_{j\in \mathbb{Z}} \Delta_j \left( (\Delta_{\le j+3} w ) (\Delta_{\ge j-3} H(u,w)) \right)\\
	&\quad +\sum_{j\in \mathbb{Z}}
			\Delta_j \left( (\Delta_{j-2\le \cdot \le j+3} w ) (\Delta_{< j-3} H(u,w)) \right) \\
	&=: G_{11}(w, H(u,w)) + G_{12}(w, H(u,w)).
\end{align*}%
Here we remark that
$2^{j-1}<|\xi| \le 2^j$, $|\xi -\eta| \le 2^{j+3}$, and $|\eta| \le 2^{j-3}$
imply
$2^{j-2}\le |\xi - \eta| \le 2^{j+3}$.
We estimate $G_{11}(w, H(u,w))$ in the same way as \eqref{4_g2H_est} and have
\begin{align*}%
	\| G_{11}(w, H(u,w)) \|_{L_t^2 (I_0 ; B^{\frac{1}{d-1}}_{\frac{2(d-1)}{d+1},2} (\mathbb{R}^d))}
	&\lesssim
	\| w \|_{L_t^2 (I_0 ; B^{-\frac{1}{d-1}}_{\frac{2(d-1)}{d-3},2}(\mathbb{R}^d))}
	\| w \|_{L_t^{\infty}(I_0 ; H^1(\mathbb{R}^d))}^{\frac{2}{d-2}} \\
\notag
	&\quad \times
	\left( \| w \|_{L_t^{\infty}(I_0 ; H^1(\mathbb{R}^d))}^{\frac{2}{d-2}}
	+ \| u \|_{L_t^{\infty}(I_0 ; H^1(\mathbb{R}^d))}^{\frac{2}{d-2}} \right).
\end{align*}%
Next, we estimate
$G_{12}(w, H(u,w))$.
By the definition of the Besov space and
the H\"{o}lder inequality with
$\frac{d+1}{2(d-1)} = \frac{d^3-3d^2 + 6d -2}{2(d^3-d^2)} + \frac{2d-1}{d^2}$,
we deduce
\begin{align*}%
	&\| G_{12}(w, H(u,w)) \|_{L_t^2 (I_0 ; B^{\frac{1}{d-1}}_{\frac{2(d-1)}{d+1},2}(\mathbb{R}^d))} \\
	&\lesssim
	\left\| \left( \sum_{j\in\mathbb{Z}}
		\left( 2^{\frac{1}{d-1}j} \| \Delta_{j-2\le \cdot \le j+3} w \|_{L_x^{\frac{2(d^3-d^2)}{d^3-3d^2+6d-2}}(\mathbb{R}^d)}
			\| \Delta_{<j-3} H(u,w) \|_{L_x^{\frac{d^2}{2d-1}}(\mathbb{R}^d)}
		\right)^2 \right)^{\frac12} \right\|_{L_t^2(I_0)} \\
	&\lesssim
	\left\| 
		\| \{ 2^{(\frac{1}{d-1}+\frac{1}{d})j}
			\| \Delta_{j} w \|_{L_x^{\frac{2(d^3-d^2)}{d^3-3d^2+6d-2}}(\mathbb{R}^d)}
		\} \|_{l^2_j}
		\| \{ 2^{-\frac{1}{d}j} \| \Delta_{<j-3} H(u,w) \|_{L_x^{\frac{d^2}{2d-1}}(\mathbb{R}^d)} \} \|_{l^{\infty}_j}
		\right\|_{L_t^2(I_0)} \\
	&\lesssim
	\left\| \| w \|_{B^{\frac{1}{d-1}+\frac{1}{d}}_{\frac{2(d^3-d^2)}{d^3-3d^2+6d-2},2}(\mathbb{R}^d)}
		\| H(u,w) \|_{B^{-\frac{1}{d}}_{\frac{d^2}{2d-1}, \infty}(\mathbb{R}^d)}
		\right\|_{L_t^2(I_0)}.
\end{align*}%
Here, for the second inequality we have used that
$2^{\frac{1}{d-1}j} = 2^{(\frac{1}{d-1}+\frac{1}{d})j} 2^{-\frac{1}{d}j}$
and the H\"{o}lder inequality for the variable $j$,
and for the third inequality we have used
Lemma \ref{lem_41} for the term
$\| \{ 2^{-\frac{1}{d}j} \| \Delta_{<j-3} H(u,w) \} \|_{L_x^{\frac{d^2}{2d-1}}(\mathbb{R}^d)} \|_{l^{\infty}_j}$.
Moreover, by the Sobolev embedding, we estimate
\begin{align*}%
	\| H(u,w) \|_{B^{-\frac{1}{d}}_{\frac{d^2}{2d-1}, \infty}(\mathbb{R}^d)}
	&\lesssim \| H(u,w) \|_{L_x^{\frac{d}{2}}(\mathbb{R}^d)} \\
	&\lesssim \| w \|_{L_x^{\frac{2d}{d-2}}(\mathbb{R}^d)}^{\frac{4}{d-2}} \\
	&\lesssim \| w \|_{B^{\frac{1}{d-1}+\frac{1}{d}}_{\frac{2(d^3-d^2)}{d^3-3d^2+6d-2},2}(\mathbb{R}^d)}^{\frac{4}{d-2}}.
\end{align*}%
Plugging it into the previous inequality, we have
\begin{align*}%
	\| G_{12}(w, H(u,w)) \|_{L_t^2 (I_0 ; B^{\frac{1}{d-1}}_{\frac{2(d-1)}{d+1},2}(\mathbb{R}^d))}
	\lesssim
	\left\|
	\| w \|_{B^{\frac{1}{d-1}+\frac{1}{d}}_{\frac{2(d^3-d^2)}{d^3-3d^2+6d-2},2}(\mathbb{R}^d)}^{\frac{d+2}{d-2}}
	\right\|_{L_t^2(I_0)}.
\end{align*}%
By the Gagliardo--Nirenberg interpolation inequality
\begin{align*}%
	\| w \|_{B^{\frac{1}{d-1}+\frac{1}{d}}_{\frac{2(d^3-d^2)}{d^3-3d^2+6d-2},2}(\mathbb{R}^d)}
	\lesssim
	\| w \|_{B^{-\frac{1}{d-1}}_{\frac{2(d-1)}{d-3}, 2}(\mathbb{R}^d)}^{1+\frac{1}{d^2} - \frac{3}{d}}
	\| w \|_{H^1(\mathbb{R}^d)}^{\frac{3}{d}-\frac{1}{d^2}},
\end{align*}%
we estimate
\begin{align*}%
	\| G_{12}(w, H(u,w)) \|_{L_t^2 (I_0 ; B^{\frac{1}{d-1}}_{\frac{2(d-1)}{d+1},2}(\mathbb{R}^d))} 
	&\lesssim \left\|
		\| w \|_{B^{-\frac{1}{d-1}}_{\frac{2(d-1)}{d-3}, 2}(\mathbb{R}^d)}^{(1+\frac{1}{d^2} - \frac{3}{d})\frac{d+2}{d-2}}
	\| w \|_{H^1(\mathbb{R}^d)}^{(\frac{3}{d}-\frac{1}{d^2})\frac{d+2}{d-2}}
	\right\|_{L^2_t(I_0)}.
\end{align*}
When $d \ge 5$,
we see that
$(1+\frac{1}{d^2} - \frac{3}{d})\frac{d+2}{d-2} > 1$
holds,
and hence,
the Sobolev embedding
$H^1(\mathbb{R}^d) \subset B^{-\frac{1}{d-1}}_{\frac{2(d-1)}{d-3}, 2}(\mathbb{R}^d)$
implies
\begin{align*}
	\| G_{12}(w, H(u,w)) \|_{L_t^2 (I_0 ; B^{\frac{1}{d-1}}_{\frac{2(d-1)}{d+1},2}(\mathbb{R}^d))}
	&\lesssim
		\| w \|_{L_t^2(I_0 ; B^{-\frac{1}{d-1}}_{\frac{2(d-1)}{d-3}, 2}(\mathbb{R}^d))}
		\| w \|_{L_t^{\infty}(I_0 ; H^1(\mathbb{R}^d))}^{\frac{4}{d-2}}.
\end{align*}%

Taking all estimates into account, we have
\begin{align*}%
	&\| w \|_{L_t^2(I_0 ; B^{-\frac{1}{d-1}}_{\frac{2(d-1)}{d-3}, 2}(\mathbb{R}^d))}\\
	&\lesssim
		\| w \|_{L_t^2(I_0 ; B^{-\frac{1}{d-1}}_{\frac{2(d-1)}{d-3}, 2}(\mathbb{R}^d))} \\
	&\quad \times
		\left( \| w \|_{L_t^{\infty}(I_0 ; H^1(\mathbb{R}^d))}^{\frac{4}{d-2}}
			+ \| u \|_{L_t^{\frac{2(d+1)}{d-2}}(I_0 ;
						B^{\frac{d}{2(d-1)}}_{\frac{2(d^2-1)}{d^2-2d+3},2}(\mathbb{R}^d))}^{\frac{4}{d-2}}
			+ \| w \|_{L_t^{\infty}(I_0 ; H^1(\mathbb{R}^d))}^{\frac{2}{d-2}}
				\| u \|_{L_t^{\infty}(I_0 ; H^1(\mathbb{R}^d))}^{\frac{2}{d-2}}
		\right).
\end{align*}%
Here, we note that
$w \in C(I_0 ; H^1(\mathbb{R}^d))$
and $w(0) = 0$,
which implies
$\| w \|_{L_t^{\infty}(I_0 ; H^1(\mathbb{R}^d))} \to 0$
as $I_0$ becomes smaller.
By \eqref{4_sol_est}, we also have
\[
	\| u \|_{L_t^{\frac{2(d+1)}{d-2}}(I_0 ;
	B^{\frac{d}{2(d-1)}}_{\frac{2(d^2-1)}{d^2-2d+3},2}(\mathbb{R}^d))} \to 0
\]
as $I_0$ becomes smaller.
Therefore, taking $I_0$ sufficiently small, we see that
$w = 0$ holds on $I_0$.
Continuing this argument, we have
$w = 0$ on the whole interval $I$.
This completes the proof for the case $d \ge 5$.

For the case $d=4$, we first note that
\begin{align*}
	u^3 - v^3 = w^2 (-3u +w) + 3wu^2.
\end{align*}
In the same way to \eqref{4_w_str}, we have
\begin{align}%
\label{4_w_str_4d}
	\| w \|_{L^2_t (I_0 ; B^{-\frac{1}{3}}_{6, 2} (\mathbb{R}^4))}
	&\lesssim
		\| 3wu^2 \|_{L_t^{\frac{20}{9}} (I_0 ;
						B^{-\frac{1}{3}}_{\frac{30}{17},2}(\mathbb{R}^4))}
	+
	\| w^2(-3u+w) \|_{L_t^{2} (I_0 ;
						B^{\frac{1}{3}}_{\frac{6}{5},2}(\mathbb{R}^d))}
\end{align}%
The term
$\| 3wu^2 \|_{L_t^{\frac{20}{9}} (I_0 ;
B^{-\frac{1}{3}}_{\frac{30}{17},2}(\mathbb{R}^4))}$
can be estimated in the same manner as
$wF'(u)$ in $d \ge 5$,
because we did not use the condition $d\ge 5$ to handle it.
Therefore, we have
\begin{align}
\label{4_w_est1_4d}
	\| 3wu^2 \|_{L_t^{\frac{20}{9}} (I_0 ;
						B^{-\frac{1}{3}}_{\frac{30}{17},2}(\mathbb{R}^4))}
	&\lesssim
	\| w \|_{L_t^2(I_0 ; B^{-\frac{1}{3}}_{6,2}(\mathbb{R}^4))}
		\| u \|_{L_t^{5}(I_0 ; B^{\frac{2}{3}}_{\frac{30}{11}, 2}(\mathbb{R}^4))}^{2}
\end{align}

For the term
$\| w^2(-3u+w) \|_{L_t^{2} (I_0 ;B^{\frac{1}{3}}_{\frac{6}{5},2}(\mathbb{R}^4))}$,
by the fractional Leibniz rule and the Sobolev embedding, we have
\begin{align}
\label{4_w_est2_1_4d}
	\| w^2(-3u+w) \|_{B^{\frac{1}{3}}_{\frac{6}{5},2}(\mathbb{R}^4)}
	&\lesssim
	\| w^2 \|_{B^{\frac{1}{3}}_{\frac{12}{7},2}(\mathbb{R}^4)}
		\| -3u + w\|_{B^0_{4,2}(\mathbb{R}^4)} \\
\notag
	&\quad + \| w^2 \|_{B^0_{2,2}(\mathbb{R}^4)} \| -3u+w \|_{B^{\frac{1}{3}}_{3,2}(\mathbb{R}^4)} \\
\notag
	&\lesssim
		\| w^2 \|_{B^{\frac{1}{3}}_{\frac{12}{7},2}(\mathbb{R}^4)}
		\| -3u +w \|_{H^1(\mathbb{R}^4)}.
\end{align}
Next, we show the product estimate
\begin{align}
\label{4_w_pro}
	\| w^2 \|_{B^{\frac{1}{3}}_{\frac{12}{7},2}(\mathbb{R}^4)}
	\lesssim
	\| w \|_{H^1(\mathbb{R}^4)}
	\| w \|_{B^{-\frac{1}{3}}_{6,2}(\mathbb{R}^4)}.
\end{align}
Indeed, applying the paraproduct
$w^2 = G_1(w,w) + G_2(w,w)$,
we have
\begin{align*}
	\| G_j (w,w) \|_{B^{\frac{1}{3}}_{\frac{12}{7},2}(\mathbb{R}^4)}
		&\lesssim
		\| w \|_{B^{-\frac{1}{3}}_{6,2}(\mathbb{R}^4)}
		\| w \|_{B^{\frac{2}{3}}_{\frac{12}{5},\infty}(\mathbb{R}^4)} \\
		&\lesssim
		\| w \|_{B^{-\frac{1}{3}}_{6,2}(\mathbb{R}^4)}
		\| w \|_{H^1(\mathbb{R}^4)}
\end{align*}
for $j=1,2$.
Here, we have used \eqref{g2_est2} and \eqref{g1_est2} in Lemma \ref{lem_41}, respectively,  and
the Sobolev embedding. This proves \eqref{4_w_pro}.
Combining \eqref{4_w_est2_1_4d} and \eqref{4_w_pro}, we see
\begin{align*}
	\| w^2(-3u+w) \|_{B^{\frac{1}{3}}_{\frac{6}{5},2}(\mathbb{R}^4)}
	&\lesssim
	\| w \|_{B^{-\frac{1}{3}}_{6,2}(\mathbb{R}^4)}
	\| w \|_{H^1(\mathbb{R}^4)}
	\| -3u + w\|_{B^0_{4,2}(\mathbb{R}^4)}.
\end{align*}
It follows from the H\"{o}lder inequality that
\begin{align}
\label{4_w_est2_2_4d}
	&\| w^2(-3u+w) \|_{L^2_t(I_0; B^{\frac{1}{3}}_{\frac{6}{5},2}(\mathbb{R}^4))} \\
\notag
	&\lesssim
	\| w \|_{L^2_t(I_0; B^{-\frac{1}{3}}_{6,2}(\mathbb{R}^4))}
	\| w \|_{L^{\infty}_t(I_0; H^1(\mathbb{R}^4))}
	\| -3u + w\|_{L^{\infty}_t(I_0; B^0_{4,2}(\mathbb{R}^4))}.
\end{align}

Applying the estimates \eqref{4_w_est1_4d} and \eqref{4_w_est2_2_4d} to
\eqref{4_w_str_4d}, we conclude
\begin{align*}
	\| w \|_{L^2_t (I_0 ; B^{-\frac{1}{3}}_{6, 2} (\mathbb{R}^4))}
	&\lesssim 
	\| w \|_{L^2_t(I_0; B^{-\frac{1}{3}}_{6,2}(\mathbb{R}^4))} \\
	&\times \left(
		\| w \|_{L^{\infty}_t(I_0; H^1(\mathbb{R}^4))}
	\| -3u + w\|_{L^{\infty}_t(I_0; B^0_{4,2}(\mathbb{R}^4))}
	+ \| u \|_{L_t^{5}(I_0 ; B^{\frac{2}{3}}_{\frac{30}{11}, 2}(\mathbb{R}^4))}^{2}
	\right)
\end{align*}
From this estimate, in the same way as in the case $d\ge 5$,
we conclude
$w=0$
and complete the proof.
\end{proof}


\appendix


\section{Besov type Strichartz estimates}
\label{secA}

Using the following lemma, we get the Besov type Strichartz estimates, Proposition \ref{prop1.5} and \ref{prop1.6}. We only give the proof of the lemma and omit the proof from the lemma to Proposition \ref{prop1.5} and \ref{prop1.6}.

\begin{lemma}
Let $s \in \R$. Assume that $(q,r)$ satisfies the assumptions in Proposition \ref{prop1.1}. Let $\gamma$ be as in Proposition \ref{prop1.1}. Then we have the following Besov type homogeneous Strichartz estimates.
\begin{align*}
	\norm{\cD(t) f}_{L^q(I:B^{s}_{r,2})} \cleq \norm{\jbra{\nabla}^{\gamma-1} f }_{B^{s}_{2,2}} \ceq \norm{f}_{B^{s+\gamma-1}_{2,2}},
\end{align*}
where  $I \subset \R$ is a time interval.
\end{lemma}

\begin{proof}
By the definition of Besov space, we have
\begin{align}
\label{eq2.5}
	&\norm{\cD(t) f}_{L^q(I:B^{s}_{r,2})} 
	\\ \notag
	&= \norm{ \norm{\cD(t) \Delta_{\leq 0} f}_{L^r} + \l\{ \sum_{j\geq1} \l( 2^{sj} \norm{\Delta_j \cD(t) f}_{L^2} \r)^2 \r\}^{\frac{1}{2}} }_{L^q(I)}
	\\ \notag
	& \leq  \norm{ \norm{\cD(t) \Delta_{\leq 0} f}_{L^r}  }_{L^q(I)}
	+ \norm{  \l\{ \sum_{j\geq1} \l( 2^{sj} \norm{\Delta_j \cD(t) f}_{L^2} \r)^2 \r\}^{\frac{1}{2}} }_{L^q(I)}.
\end{align}
By the homogeneous Strichartz estimates in Proposition \ref{prop1.1}, the first term can be estimated as
\begin{align}
\label{eq2.6}
	\norm{ \norm{\cD(t) \Delta_{\leq 0} f}_{L^r}  }_{L^q(I)}
	\cleq \norm{ \jbra{\nabla}^{\gamma-1} \Delta_{\leq 0} f}_{L^2} 
	=  \norm{ \Delta_{\leq 0} \jbra{\nabla}^{\gamma-1} f}_{L^2}.
\end{align}
The second term is estimated as
\begin{align}
\label{eq2.7}
	\norm{  \l\{ \sum_{j\geq1} \l( 2^{sj} \norm{\Delta_j \cD(t) f}_{L^2} \r)^2 \r\}^{\frac{1}{2}} }_{L^q(I)}
	&= \norm{   \sum_{j\geq1} \l( 2^{sj} \norm{\Delta_j \cD(t) f}_{L^2} \r)^2  }_{L^{\frac{q}{2}}(I)}^{\frac{1}{2}}
	\\ \notag
	&\leq \sum_{j\geq1}  \norm{  \l( 2^{sj} \norm{\Delta_j \cD(t) f}_{L^2} \r)^2  }_{L^{\frac{q}{2}}(I)}^{\frac{1}{2}}
	\\ \notag
	&=\l\{ \sum_{j\geq1}  \l( 2^{sj} \norm{\Delta_j \cD(t) f}_{L^q(I:L^2)} \r)^2  \r\}^{\frac{1}{2}}
	\\ \notag
	&\cleq  \l\{ \sum_{j\geq1}  \l( 2^{sj} \norm{\Delta_j \jbra{\nabla}^{\gamma-1} f}_{L^2} \r)^2  \r\}^{\frac{1}{2}}.
\end{align}
Combining \eqref{eq2.6} and \eqref{eq2.7} with \eqref{eq2.5}, we get 
\begin{align*}
	\norm{\cD(t) f}_{L^q(I:B^{s}_{r,2})}  
	&\cleq \norm{ \Delta_{\leq 0} \jbra{\nabla}^{\gamma-1} f}_{L^2} +  \l\{ \sum_{j\geq1}  \l( 2^{sj} \norm{\Delta_j \jbra{\nabla}^{\gamma-1} f}_{L^2} \r)^2  \r\}^{\frac{1}{2}}
	\\
	& = \norm{\jbra{\nabla}^{\gamma-1} f }_{B^{s}_{2,2}}. 
\end{align*}
The last equivalency is a fundamental property of Besov spaces. 
\end{proof}

The estimates for $\partial_t \cD(t)$ and $\partial_t^2 \cD(t)$ can be obtained in the same way as above. Thus, we get Proposition \ref{prop1.5}.

\begin{lemma}
Let $s \in \R$. Assume that $(q,r)$ and $(\tilde{q},\tilde{r})$ satisfy the assumptions in Proposition \ref{prop1.4}. Let $\gamma$ and $\tilde{\gamma}$ be as in Proposition \ref{prop1.4} and $\delta$ be defined in Table \ref{tab1}. Then we have the following Besov type inhomogeneous Strichartz estimates.
\begin{align*}
	\norm{ \int_{0}^{t} \cD(t-s) F(s) ds}_{L^q(I:B^{s}_{r,2})} 
	\cleq \norm{\jbra{\nabla}^{\gamma+\tilde{\gamma} + \delta-1} F }_{L^{\tilde{q}'}(I:B^{s}_{\tilde{r}',2})}
\end{align*}
where $I \subset \R$ is a time interval.
\end{lemma}

\begin{proof}
For simplicity, we set $I=\int_{0}^{t} \cD(t-s) F(s) ds$ and $G= \jbra{\nabla}^{\gamma+\tilde{\gamma} + \delta-1} F$. 
By the definition of Besov spaces and $(a^2)^{1/2} + (\sum_j b_j^2)^{1/2} \leq \sqrt{2} (a^2+\sum_j b_j^2)^{1/2}$, we have
\begin{align*}
	(\text{L.H.S}) 
	&=\norm{ \norm{\Delta_{\leq0} I }_{L^r} + \l\{  \sum_{j \geq1} \l(  2^{sj} \norm{\Delta_{j} I}_{L^r} \r)^2   \r\}^{1/2}  }_{L^q(t)}
	\\
	&\cleq \norm{  \l\{ \norm{\Delta_{\leq0} I }_{L^r}^2 +   \sum_{j \geq1} \l(  2^{sj} \norm{\Delta_{j} I}_{L^r} \r)^2   \r\}^{1/2}  }_{L^q(t)}.
\end{align*}
Since $q \geq2$, the right hand side is estimated as
\begin{align*}
	(\text{R.H.S}) 
	\cleq  \l\{ \norm{\Delta_{\leq0} I }_{L^qL^r}^2 +   \sum_{j \geq1} \l(  2^{sj} \norm{\Delta_{j} I}_{L^qL^r} \r)^2   \r\}^{1/2}.
\end{align*}
By using the inhomogeneous Strichartz estimates, Proposition \ref{prop1.4}, we get
\begin{align*}
	&\l\{ \norm{\Delta_{\leq0} I }_{L^qL^r}^2 +   \sum_{j \geq1} \l(  2^{sj} \norm{\Delta_{j} I}_{L^qL^r} \r)^2   \r\}^{1/2}
	\\
	&\cleq \l\{ \norm{\Delta_{\leq0} G }_{L^{\tilde{q}'}L^{\tilde{r}'}}^2 +   \sum_{j \geq1} \l(  2^{sj} \norm{\Delta_{j} G}_{L^{\tilde{q}'}L^{\tilde{r}'}} \r)^2   \r\}^{1/2}
	\\
	&=  \l\{ \norm{ \norm{\Delta_{\leq0} G }_{L^{\tilde{r}'}}^2 }_{L^{\frac{\tilde{q}'}{2}(I)}}  +  \sum_{j \geq1} \norm{ \l(  2^{sj} \norm{\Delta_{j} G}_{L^{\tilde{r}'}} \r)^2 }_{L^{\frac{\tilde{q}'}{2}(I)}}   \r\}^{1/2}.
\end{align*}
Since $\tilde{q}' \leq 2$, this right hand side is bounded as
\begin{align*}
	(\text{R.H.S})
	&\cleq  \norm{ \norm{\Delta_{\leq0} G }_{L^{\tilde{r}'}}^2  +  \sum_{j \geq1}  \l(  2^{sj} \norm{\Delta_{j} G}_{L^{\tilde{r}'}} \r)^2 }_{L^{\frac{\tilde{q}'}{2}}(I)}  ^{1/2}
	\\
	&= \norm{ \l\{ \norm{\Delta_{\leq0} G }_{L^{\tilde{r}'}}^2  +  \sum_{j \geq1} \l(  2^{sj} \norm{\Delta_{j} G}_{L^{\tilde{r}'}} \r)^2  \r\}^{1/2} }_{L^{\tilde{q}'}(I)}
	\\
	& \cleq  \norm{ \norm{\Delta_{\leq0} G }_{L^{\tilde{r}'}}  + \l\{ \sum_{j \geq1} \l(  2^{sj} \norm{\Delta_{j} G}_{L^{\tilde{r}'}} \r)^2  \r\}^{1/2} }_{L^{\tilde{q}'}(I)}
	\\
	&= \norm{G}_{L^{\tilde{q}'} (I:B^{s}_{\tilde{r}',2})}.
\end{align*}
This completes the proof. 
\end{proof}

We can estimate $ \int_{0}^{t} \cD(t-s) F(s) ds$ in the same way. Therefore, we get Proposition \ref{prop1.6}. 


\section{Proof of Lemma \ref{lem_41}}
\label{secB}

\begin{proof}[Proof of Lemma \ref{lem_41}]

First, we consider the first equivalence. We show $\gtrsim$. It is true that $\| f \|_{B_{p,q}^{s}} \sim \| f \|_{L^{p}} + \| f \|_{\dot{B}_{p,q}^{s}}$ for $s>0$. By \cite[Lemma 4.1]{Bu13}, we have
\begin{align*}
	 \| f \|_{\dot{B}_{p,q}^{s}}
	 \sim \left\|  \{ 2^{js} \| \Delta_{\geq j} f \|_{L^p}  \}_{j\in \mathbb{Z}} \right\|_{l^{q}}.
\end{align*}
Hence, it follows that
\begin{align*}
	&\| \Delta_{\leq J} f \|_{L^p}  +  \left\|  \{ 2^{js} \| \Delta_{\geq j} f \|_{L^p}  \}_{j=J}^{\infty} \right\|_{l^{q}}
	\\
	&\lesssim 
	\|  \Delta_{\leq 0} f \|_{L^p}
		+\sum_{j=1}^{J} \| \Delta_j f \|_{L^p}
		+  \left\|  \{ 2^{js} \| \Delta_{\geq j} f \|_{L^p}  \}_{j =J}^{\infty} \right\|_{l^{q}}
	\\
	&\lesssim
	\|  \Delta_{\leq 0} f \|_{L^p}
		+ \left\|  \{ 2^{js} \| \Delta_{\geq j} \Delta_{\geq 0} f \|_{L^p}  \}_{j \in \mathbb{Z}} \right\|_{l^{q}} \\
	&\sim 
	\| \Delta_{\leq 0} f \|_{L^{p}} + \| \Delta_{\geq 0} f \|_{\dot{B}_{p,q}^{s}}
	\\
	& \sim
	\| f \|_{B_{p,q}^{s}}.
\end{align*}
We prove $\lesssim$. We have
\begin{align*}
	\| f \|_{B_{p,q}^{s}} 
	&=\| \Delta_{\leq 0} f \|_{L^p} + \left\|  \{ 2^{js} \| \Delta_j f \|_{L^p}  \}_{j=1}^{\infty} \right\|_{l^{q}}
	\\
	& \lesssim_{J}  \| \Delta_{\leq 0} f \|_{L^p}+ \sum_{j=1}^{J-1} \| \Delta_j f \|_{L^p}  + \left\|  \{ 2^{js} \| \Delta_j f \|_{L^p}  \}_{j=J}^{\infty} \right\|_{l^{q}}
	\\
	& \lesssim_{J}  \| \Delta_{\leq J} f \|_{L^p} + \left\|  \{ 2^{js} \| \Delta_j f \|_{L^p}  \}_{j=J}^{\infty} \right\|_{l^{q}}
	\\
	& \lesssim_{J}  \| \Delta_{\leq J} f \|_{L^p} + \left\|  \{ 2^{js} \| \Delta_{\geq j} f \|_{L^p}  \}_{j=J}^{\infty} \right\|_{l^{q}}.
\end{align*}
Next, we show the second equivalence. 
We first prove $\lesssim$. 
\begin{align*}
	\| f \|_{B_{p,q}^{-s}} 
	&= \| \Delta_{\leq 0} f \|_{L^p} + \left\|  \{ 2^{-js} \| \Delta_j f \|_{L^p}  \}_{j=1}^{\infty} \right\|_{l^{q}}
	\\
	&\lesssim \| \Delta_{\leq 0} f \|_{L^p} + \left\|  \{ 2^{-js} \| \Delta_{\leq j} f \|_{L^p}  \}_{j=0}^{\infty} \right\|_{l^{q}}
	\\
	&\lesssim_{J} \| \Delta_{\leq J} f \|_{L^p} + \left\|  \{ 2^{-js} \| \Delta_{\leq j} f \|_{L^p}  \}_{j=J}^{\infty} \right\|_{l^{q}}
\end{align*}
where we use $ \Delta_j =  \Delta_{\leq j} - \Delta_{\leq j-1}$ and the triangle inequality in the first inequality. We show $\gtrsim$. Now, by the Young inequality, we have
\begin{align*}
	\left\|  \{ 2^{-js} \| \Delta_{\leq j} f \|_{L^p}  \}_{j=J}^{\infty} \right\|_{l^{q}}
	&= \left\|  \left\{ 2^{-js} \| \sum_{l\leq j}  \Delta_{l} f \|_{L^p}  \right\}_{j=J}^{\infty} \right\|_{l^{q}}
	\\
	&\lesssim  \left\|  \left\{  \sum_{l\leq j}   2^{-js} \left\|\Delta_{l} f \right\|_{L^p}  \right\}_{j=J}^{\infty} \right\|_{l^{q}}
	\\
	&\lesssim  \left\|  \left\{  \sum_{l\leq j}   2^{-(j-l)s} 2^{-ls} \left\|\Delta_{l} f \right\|_{L^p}  \right\}_{j=J}^{\infty} \right\|_{l^{q}}
	\\
	&\lesssim  \left\| \{   2^{-js} \}_{j=J}^{\infty} \right\|_{l^{1}} 
	\left\|  \{   2^{-js} \left\|\Delta_{j} f \right\|_{L^p}  \}_{j=J}^{\infty} \right\|_{l^{q}}
	\\
	&\lesssim  \left\|  \{   2^{-js} \left\| \Delta_{j} f \right\|_{L^p} \}_{j=1}^{\infty}  \right\|_{l^{q}}.
\end{align*}
Therefore, it follows that 
\begin{align*}
	&\| \Delta_{\leq J} f \|_{L^p}  +  \left\|  \{ 2^{-js} \| \Delta_{\leq j} f \|_{L^p}  \}_{j=J}^{\infty} \right\|_{l^{q}}
	\\
	&\lesssim_{J}
		\| \Delta_{\leq 0} f \|_{L^p} 
		+ \sum_{j=1}^{J} \| \Delta_{j} f \|_{L^p}
		+ \| \{ 2^{-js} \left\|\Delta_{j} f \right\|_{L^p}  \}_{j=1}^{\infty} \|_{l^q}
	\\
	& \lesssim_{J}  \| \Delta_{\leq 0} f \|_{L^p} + \| \{ 2^{-js} \left\|\Delta_{j} f \right\|_{L^p}  \}_{j=1}^{\infty} \|_{l^q}
	\\
	&\lesssim_{J} \| f \|_{B_{p,q}^{-s}}.
\end{align*}
This completes the proof. 
\end{proof}


\section{An estimate of Besov norm for H\"{o}lder continuous functions}
For the proof of unconditional uniqueness (in Section \ref{sec4}),
we prepare the following estimate of Besov norm for H\"{o}lder continuous functions.
The corresponding estimate in homogeneous Besov norms
has already given in \cite[Lemma 2.1]{Bu13}.
\begin{lemma}\label{lem_b}
Let $0<s<1$, $0<\alpha \le 1$ and $1 < p \le \infty$
be given such that 
$1 < p\alpha \le \infty$.
Let $f(z)$ be a H\"{o}lder continuous function of order $\alpha$
satisfying $f(0) = 0$.
Then, we have
\begin{align*}%
	\| f(u) \|_{B^s_{p,\infty}} \lesssim \| u \|_{B^{\frac{s}{\alpha}}_{p\alpha, \infty}}^{\alpha}.
\end{align*}%
\end{lemma}
\begin{proof}
We recall the equivalence of Besov norm (see \cite[p.162]{BeLo76})
\begin{align*}%
	\| u \|_{B^s_{p,q}(\mathbb{R}^d)} \sim
		\| u \|_{L^p(\mathbb{R}^d)}
		+ \left( \int_{\mathbb{R}^d}
				\frac{\| u(\cdot + h) - u(\cdot) \|_{L^p(\mathbb{R}^d)}}{|h|^{d+sq}} \,dh \right)^{\frac{1}{q}}
\end{align*}%
for $q < \infty$, and
\begin{align*}%
	\| u \|_{B^s_{p,q}(\mathbb{R}^d)} \sim
		\| u \|_{L^p(\mathbb{R}^d)}
		+ \sup_{h \in \mathbb{R}^d}
				|h|^{-s} \| u(\cdot + h) - u(\cdot) \|_{L^p(\mathbb{R}^d)}
\end{align*}%
for $q = \infty$.
Since $f(z)$ is H\"{o}lder continuous of order $\alpha$ and satisfies $f(0) = 0$,
we deduce
$|f(u)| \lesssim |u|^{\alpha}$
and
$|f(u(x+h)) - f(u(x))| \lesssim | u(x+h) - u(x) |^{\alpha}$.
Therefore, we calculate
\begin{align*}%
	\| f(u) \|_{B^s_{p,\infty}}
	&\lesssim
	\| f(u) \|_{L^p(\mathbb{R}^d)}
	+\sup_{h\in \mathbb{R}^d} \left[ |h|^{-s} \left( \int_{\mathbb{R}^d} |f(u(x+h))-f(u(x))|^p\,dx \right)^{\frac{1}{p}} \right]\\
	&\lesssim
	\| u \|_{L^{p\alpha}(\mathbb{R}^d)}^{\alpha}
	+\sup_{h\in \mathbb{R}^d} \left[ |h|^{-s} 
				\left( \int_{\mathbb{R}^d} |u(x+h)-u(x)|^{p\alpha} \,dx \right)^{\frac{1}{p}} \right]\\
	&\lesssim
	\| u \|_{L^{p\alpha}(\mathbb{R}^d)}^{\alpha}
	+\sup_{h\in \mathbb{R}^d} \left( |h|^{-\frac{s}{\alpha}}
					\| u(\cdot + h) - u(\cdot) \|_{L^{p\alpha}(\mathbb{R}^d)} \right)^{\alpha} \\
	&\lesssim
	\| u \|_{L^{p\alpha}(\mathbb{R}^d)}^{\alpha}
	+ \left( \sup_{h\in \mathbb{R}^d} |h|^{-\frac{s}{\alpha}}
					\| u(\cdot + h) - u(\cdot) \|_{L^{p\alpha}(\mathbb{R}^d)} \right)^{\alpha} \\
	&\lesssim
	\left( \| u \|_{L^{p\alpha}(\mathbb{R}^d)}
	+ \sup_{h\in \mathbb{R}^d} |h|^{-\frac{s}{\alpha}}
					\| u(\cdot + h) - u(\cdot) \|_{L^{p\alpha}(\mathbb{R}^d)} \right)^{\alpha} \\
	&\sim \| u \|_{B^{\frac{s}{\alpha}}_{p\alpha, \infty}(\mathbb{R}^d)},
\end{align*}%
which gives the desired estimate.
\end{proof}

\section{Failure of the endpoint Strichartz estimate when $d=3$}
In this section, we prove that the endpoint Strichartz estimate does not hold for $d=3$.
Namely, it is not true that there exists a constant
$C>0$
such that the estimate
\begin{align}
\label{end_str_3}
	\left\| \mathcal{D}(t) g \right\|_{L^2_t(\mathbb{R}_+ ; L^{\infty}_x(\mathbb{R}^3))}
		\le C \| g \|_{L^2(\mathbb{R}^3)}
\end{align}
holds for all $g \in L^2(\mathbb{R}^3)$.
In what follows we give a proof based on the argument of Klainerman--Machedon \cite{KlMa93}
with the decomposition of the solution given by Nishihara \cite{Nis03}.
To prove the failure of \eqref{end_str_3}, we show the following:
For any $n \in \mathbb{N}$, there exists a nonnegative function $g_n \in C_0^{\infty}(\mathbb{R}^3)$
such that $\| g_n \|_{L^2(\mathbb{R}^3)} = 1$ and
\begin{align}
\label{fail_end_str}
	\int_0^{\infty} | \phi_n (t,t,0,0) |^2 \,dt \ge n,
\end{align}
where $\phi_n (t,x) = \mathcal{D}(t) g_n(x)$.
Let us prove \eqref{fail_end_str} by contradiction.
Suppose that there exist a constant $C>0$
and a nonnegative test function $\varphi= \varphi(t) \in \mathcal{S}(\mathbb{R}) \setminus\{0\}$
such that
for any $g \in C_0^{\infty}(\mathbb{R}^3)$
satisfying $g\ge 0$ and $\| g \|_{L^2(\mathbb{R}^3)} = 1$, the boundedness
\begin{align}
\label{app_d_1}
	\mathcal{J} = \int_0^{\infty} \phi(t,t,0,0) \varphi(t) \,dt < C
\end{align}
holds, where $\phi (t,x) = \mathcal{D}(t) g(x)$.
Here, we recall the decomposition of the solution in $d=3$:
\begin{align*}
	\phi(t,x) = J_0(t) g + e^{-t/2} W(t) g,
\end{align*}
which was given by Nishihara \cite[(1.10)]{Nis03}, where
\begin{align*}
	J_0(t) g (x) &=
		\frac{e^{-t/2}}{8\pi} \int_0^t \int_{S^2}
			I_1\left( \frac{1}{2}\sqrt{t^2 - r^2} \right) \frac{g(x+r\omega)r^2}{\sqrt{t^2 - r^2}} \,d\omega dr, \\
	W(t) g(x) &=
		\frac{t}{4\pi} \int_{S^2} g(x+t\omega) \,d\omega,
\end{align*}
and $I_1(z)$ is the modified Bessel function of first kind defined by
\begin{align*}
	I_1(z) = \sum_{m=0}^{\infty} \frac{1}{m! (m+1)!} \left( \frac{z}{2} \right)^{2m}.
\end{align*} 
In particular, since $I_1(s) \ge 0$ for $s \in \mathbb{R}$, we have
$J_0(t) g \ge 0$ if $g \ge 0$.
Also, from the above expression, it is obvious that $W(t) g \ge 0$ if $g \ge 0$.
Thus, if $g \ge 0$, then we have
$\phi(t,x) \ge e^{-t/2} W(t) g \ge 0$.
Therefore, noting $\varphi \ge 0$, we estimate
\begin{align*}
	\mathcal{J} \ge
		\int_0^{\infty} \frac{t e^{-t/2}}{4\pi} \int_{S^2} g(t+t\omega_1, t \omega_2, t \omega_3) \,d\omega
			\cdot \varphi(t) \,dt.
\end{align*}
We put $\tilde{\varphi}(t) = e^{-t/2} \varphi(t)$
and apply the coordinate transform $t \omega = y$ to obtain
\begin{align*}
	\int_0^{\infty} \frac{t e^{-t/2}}{4\pi} \int_{S^2} g(t+t\omega_1, t \omega_2, t \omega_3) \,d\omega
			\cdot \varphi(t) \,dt
	= \frac{1}{4\pi} \int_{\mathbb{R}^3} \frac{1}{|y|} g(y_1 + |y|, y_2, y_3) \tilde{\varphi} (|y|) \,dy.
\end{align*}
We further change the variable by $z = y+(|y|, 0, 0)$ to deduce
\begin{align*}
	\frac{1}{4\pi} \int_{\mathbb{R}^3} \frac{1}{|y|} g(y_1 + |y|, y_2, y_3) \tilde{\varphi} (|y|) \,dy
	 = \frac{1}{4\pi} \int_{z_1 > 0} \frac{1}{z_1} g(z) \tilde{\varphi}\left( \frac{|z|^2}{2z_1} \right) \,dz.
\end{align*}
Combining this with \eqref{app_d_1}, we conclude that
\begin{align}
\label{app_d_2}
	\frac{1}{4\pi} \int_{z_1 > 0} \frac{1}{z_1} g(z) \tilde{\varphi}\left( \frac{|z|^2}{2z_1} \right) \,dz
		\le \mathcal{J} < C
\end{align}
holds for any
$g \in C_0^{\infty}(\mathbb{R}^3)$ satisfying $g\ge 0$ and $\| g \|_{L^2(\mathbb{R}^3)} = 1$.
Noting $\tilde{\varphi} \ge 0$, we see that
the function
\begin{align*}
	\psi(z) = \frac{1}{z_1} \tilde{\varphi}\left( \frac{|z|^2}{2z_1} \right)
\end{align*}
is also nonnegative on $\mathbb{R}^3_+ = \{ z \in \mathbb{R}^3 ; z_1 > 0 \}$.
Therefore, by \eqref{app_d_2}, we have
\begin{align*}
	\| \psi \|_{L^2(\mathbb{R}^3_+)}
		= \sup\left\{ (\psi, g)_{L^2(\mathbb{R}^3_+)} ;
			g \in C_0^{\infty}(\mathbb{R}^3_+), g \ge 0, \| g \|_{L^2(\mathbb{R}^3_+)} \le 1 \right\}
		\le 4\pi C,
\end{align*}
that is, $\psi \in L^2(\mathbb{R}^3_+)$.
Moreover, we easily compute
\begin{align*}
	\| \psi \|_{L^2(\mathbb{R}^3_+)}^2
		&= \int_{\mathbb{R}^3_+} \frac{1}{z_1^2} \tilde{\varphi} \left( \frac{|z|^2}{2z_1} \right)^2 \,dz \\
		&= \int_{\mathbb{R}^3} \frac{1}{(y_1+|y|)|y|} \tilde{\varphi}(|y|)^2 \,dy \\
		&= 2\pi \int_0^{\infty} \tilde{\varphi}(\lambda)^2 \int_0^{\pi} \frac{\sin \theta}{1+\cos \theta} \,d\theta,
\end{align*}
where we use the changing variable
$z = y+(|y|,0,0)$
and the polar coordinates
$y = (\lambda \cos \theta, \lambda \sin \theta \cos \mu, \lambda \sin \theta \sin \mu)$.
However, the integral of the right-hand side does not converge unless $\tilde{\varphi} \equiv 0$.
This is a contradiction and hence, we have proved the failure of \eqref{app_d_1}.
Therefore, for any $C>0$ and any nonnegative $\varphi \in \mathcal{S}(\mathbb{R}^3)\setminus \{0\}$,
there exists $g \in C_0^{\infty}(\mathbb{R}^3)$ with $g\ge 0$ and $\| g\|_{L^2(\mathbb{R}^3)} = 1$
such that
\begin{align*}
	C &\le \int_0^{\infty} \phi(t,t,0,0) \varphi(t) \,dt \\
		&\le 2 \int_0^{\infty} |\phi(t,t,0,0)|^2 \,dt + 2 \int_0^{\infty} |\varphi(t)|^2 \,dt.
\end{align*}
For each $n \in \mathbb{N}$,
if we choose
$C = 2n + 1$ and
$\varphi \in \mathcal{S}(\mathbb{R}^3)\setminus \{0\}$
satisfying $\varphi \ge 0$ and $\| \varphi \|_{L^2(\mathbb{R}_+)}^2 = 1/2$, 
then we can take $g_n \in C_0^{\infty}(\mathbb{R}^3)$ with $g_n \ge 0$ and $\| g_n \|_{L^2(\mathbb{R}^3)} = 1$
satisfying \eqref{fail_end_str}.

\section*{Acknowledgment}
This work was supported by JSPS KAKENHI Grant Numbers JP16K17625, 18K13444.


\end{document}